\documentclass[hidelinks,11pt]{article}
\usepackage{graphicx,wrapfig,psfrag,hyperref,amsthm}
\usepackage[center]{caption}
\usepackage{amssymb,flafter,indentfirst,color}
\usepackage{mathtools}

\newtheorem{theorem}{Theorem}[section]
\newtheorem{remark}{Remark}[section]
\newtheorem{lemma}{Lemma}[section]
\newcommand{\aver}[1]{\left\{\!\!\left\{#1\right\}\!\!\right\}}

\newcommand{\jump}[1]{\left[\!\left[#1\right]\!\right]}
\newcommand{\bigjump}[1]{\left[\!\!\left[#1\right]\!\!\right]}
\begin{document}
\arraycolsep=1pt
\baselineskip 11pt
\title{\Large\bf Partially Penalized Immersed Finite Element Methods \\
for Parabolic Interface Problems
\thanks{This research is partially supported by the NSF grant DMS-1016313}
\thanks{The second author is supported by a project of Shandong province higher educational science and technology program (J14LI03), P.R.China}}
\author{Tao Lin\thanks{Department of Mathematics, Virginia Tech, Blacksburg, VA
24061, tlin@math.vt.edu},\ \ Qing Yang\thanks{School of Mathematical Science, Shandong
Normal University, Jinan 250014, P. R. China, sd\_yangq@163.com}\ \ and Xu Zhang\thanks{Department of Mathematics,
Purdue University, West Lafayette, IN, 47907, xuzhang@purdue.edu}}
\date{}
\maketitle
\baselineskip 15pt
\begin{center}
\begin{minipage}{150mm}{\small
{\bf Abstract}\hspace{2mm} We present partially penalized immersed finite
element methods for solving parabolic interface problems on Cartesian meshes.
Typical semi-discrete and fully discrete schemes are discussed. Error estimates in an energy norm are derived. Numerical examples are provided to support theoretical analysis.
\par
 \vspace{3mm}
 {\bf Key words:}  parabolic interface problems, Cartesian mesh methods, partially penalized immersed finite element, error estimation.
\par
\vspace{3mm}
{\bf 2010 Mathematics Subject Classifications
}\hspace{2mm}65M15,\hspace{2mm}65M60}
\end{minipage}
\end{center}

\section{Introduction}
\par
\noindent
In this article, we consider the following parabolic equation with the Dirichlet boundary condition
\begin{eqnarray}
&&\frac{\partial u}{\partial t}-\nabla\cdot(\beta\nabla u)=f(X,t),\ \ X = (x,y)\in\Omega^+\cup\Omega^-,\ t\in  (0,T], \label{eq:parab_eq}\\
&&u|_{\partial\Omega}=g(X,t),\ \ t\in  (0,T], \label{eq:parab_eq_bc}\\
&&u|_{t=0}=u_0(X),\ \ X\in\Omega. \label{eq:parab_eq_ic}
\end{eqnarray}
Here, $\Omega$ is a rectangular domain or a union of several rectangular domains in $\mathbb{R}^2$. The interface $\Gamma\subset \overline{\Omega}$ is a smooth curve separating $\Omega$ into two sub-domains $\Omega^-$ and $\Omega^+$ such that $\overline\Omega=\overline{\Omega^-\cup\Omega^+\cup\Gamma}$, see the left plot in Figure \ref{fig: domain}.
The diffusion coefficient $\beta$ is discontinuous across the interface, and it is assumed to be a piecewise constant function such that
\[
\beta(X)=\left\{
\begin{array}{ll}
\beta^-,&~~ X \in\Omega^-, \\
\beta^+,&~~ X \in\Omega^+,
\end{array}
\right.
\]
and min$\{\beta^-,\beta^+\}>0$. We assume that the exact solution $u$ to the above initial boundary value problem
satisfies the following jump conditions across the interface $\Gamma$:
\begin{eqnarray}
 &&\jump{u}_{\Gamma}=0, \label{eq:parab_eq_jump_1} \\
 &&\bigjump{\beta\frac{\partial u}{\partial \mathbf{n}}}_{\Gamma}=0. \label{eq:parab_eq_jump_2}
\end{eqnarray}

\begin{figure}[htbp]
\centering
\includegraphics[width=.3\textwidth]{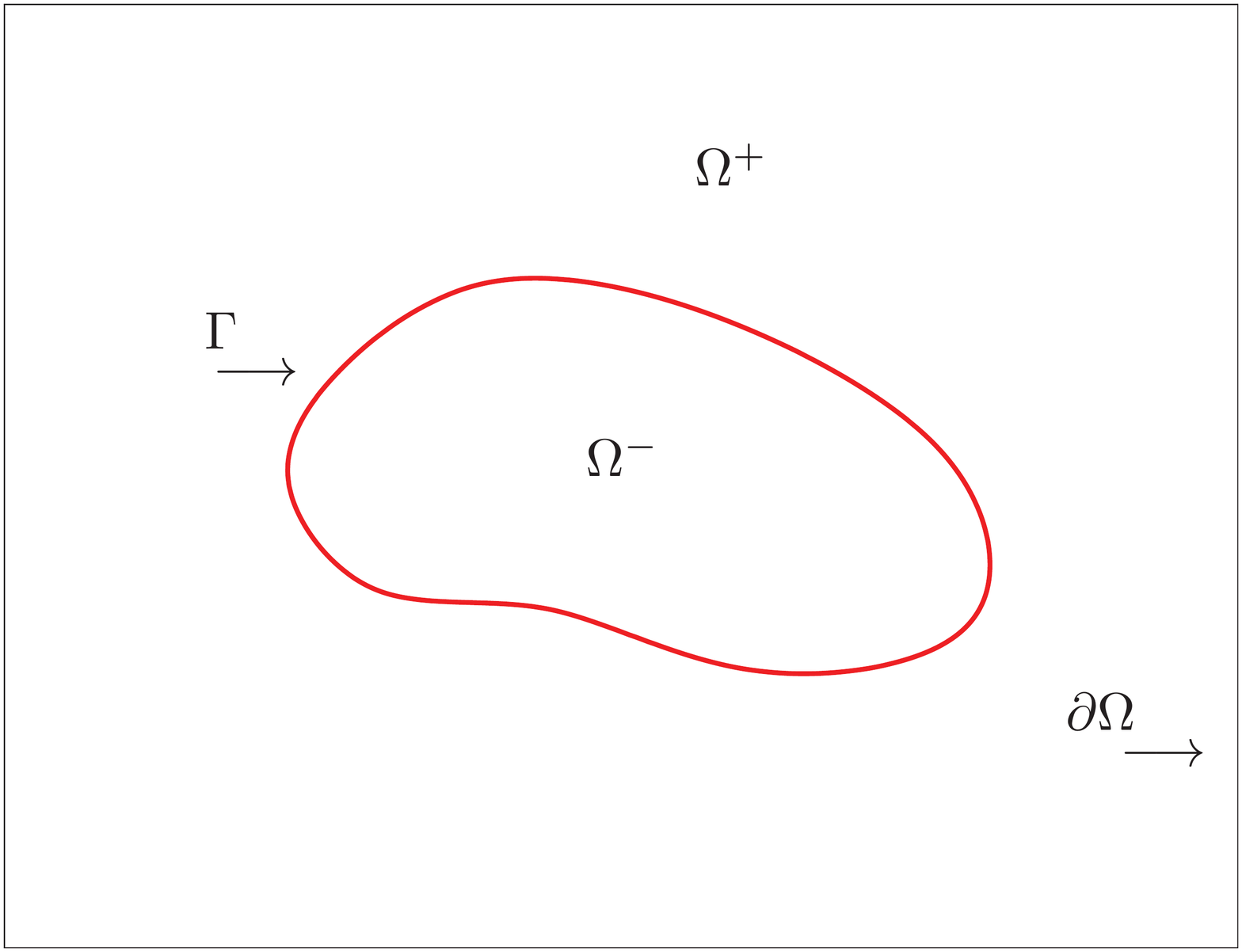}
\includegraphics[width=.3\textwidth]{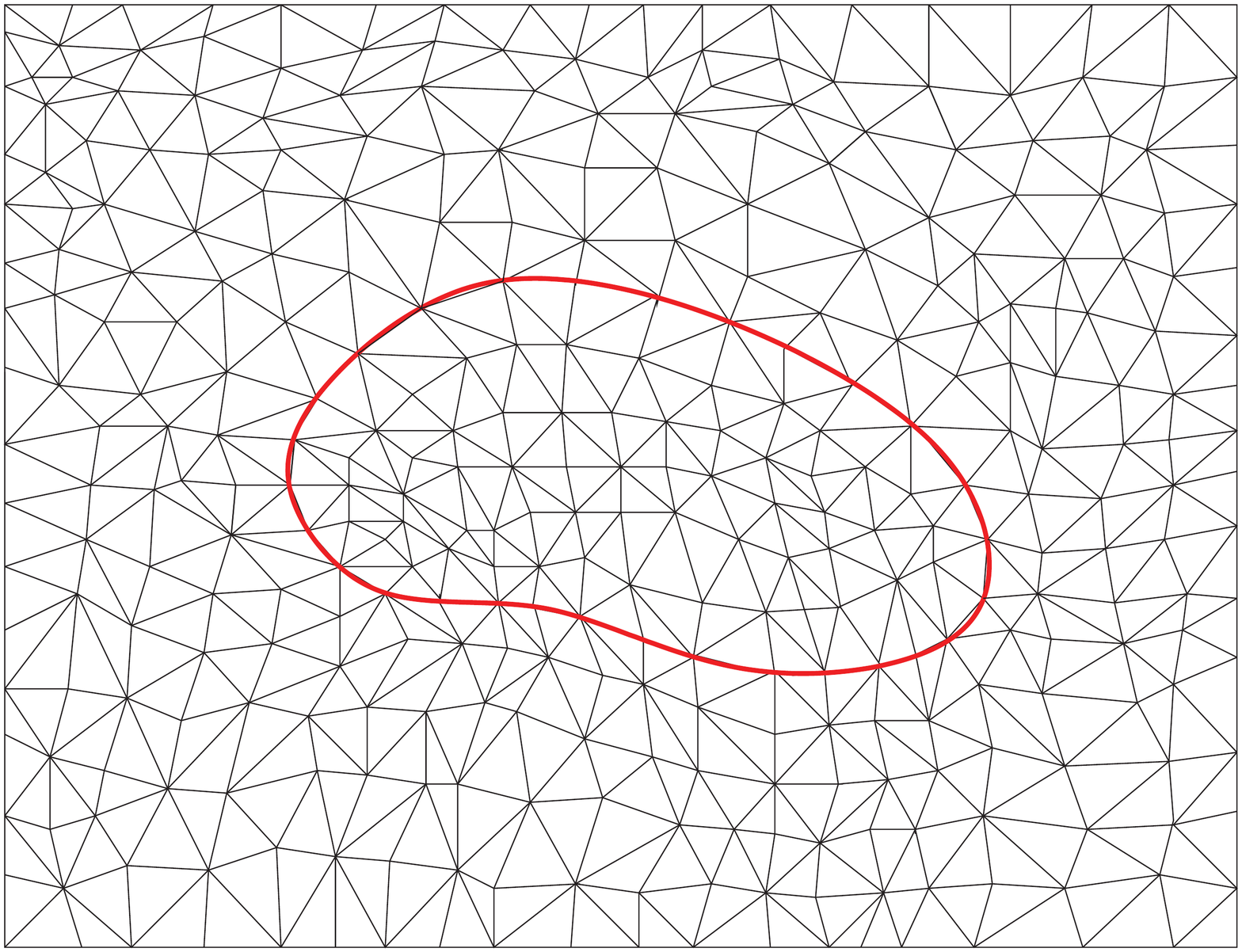}
\includegraphics[width=.3\textwidth]{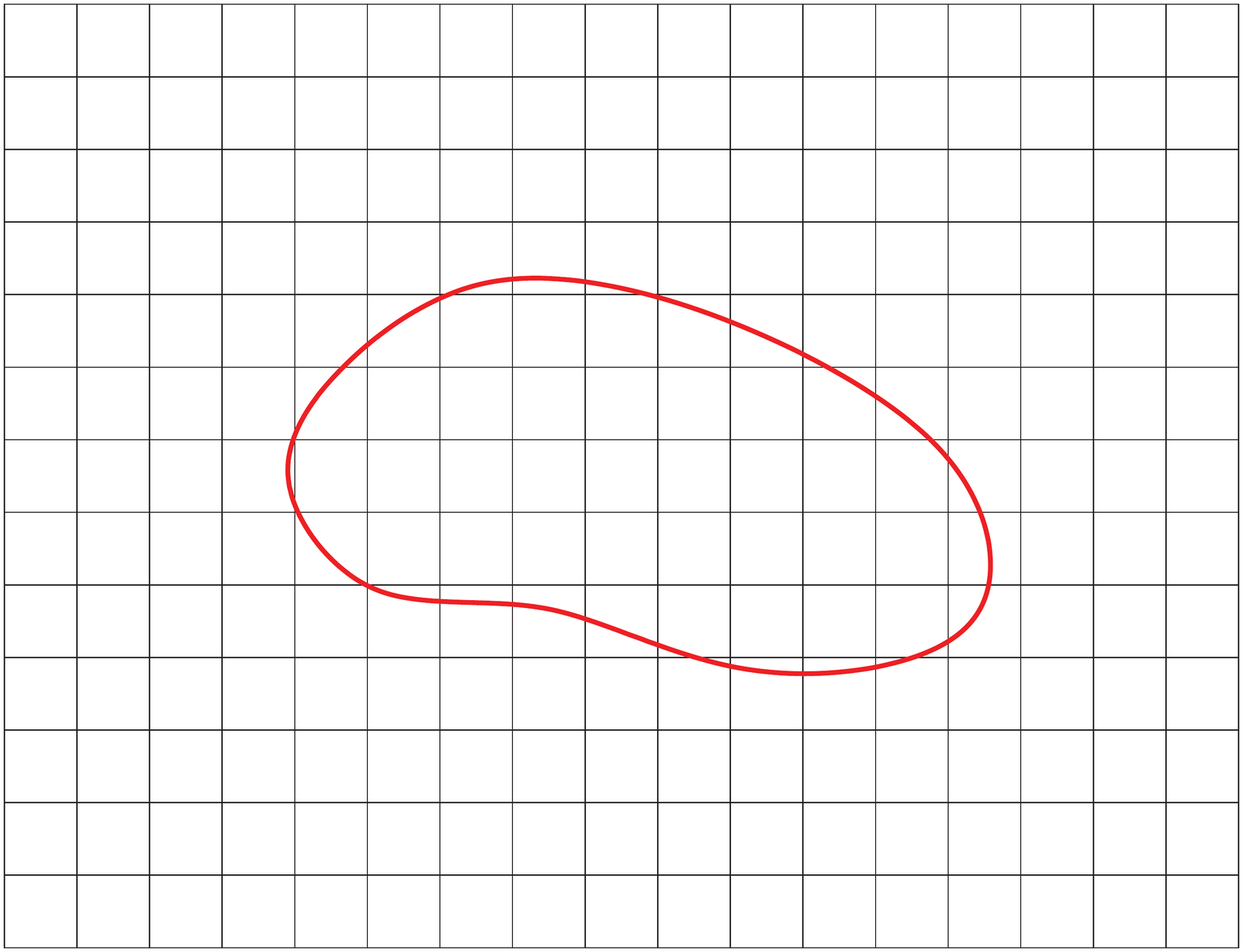}
\caption{\small The simulation domain $\Omega$ (left), body-fitting mesh (middle), and non-body-fitting mesh (right).}
\label{fig: domain}
\end{figure}

Interface problems appear in many applications of engineering and science; therefore, it is of great importance to solve interface problems efficiently. When conventional finite element methods are employed to solve interface problems, body-fitting meshes (see the mid plot in Figure \ref{fig: domain}) have to be used in order to guarantee their optimal convergence
\cite{
Babuska_Elliptic_Discontinuous,
Babuska_Osborn_Bad_FEM,
JBramble_JKing_FEM_Interface,
ZChen_JZou_FEM_Elliptic}.
Such a restriction hinders their applications in some situations because it prevents the use of Cartesian mesh unless the interface has a very simple geometry such as an axis-parallel straight line. Recently, immersed finite element (IFE) methods have been developed to overcome such a limitation of traditional finite element methods for solving interface problems, see \cite{
Chou_Kwak_Wee_IFE_Triangle_Analysis,
XHe_Thesis_Bilinear_IFE,
XHe_TLin_YLin_Bilinear_Approximation,
XHe_TLin_YLin_Convergence_IFE,
ZLi_IIM_FE,
ZLi_TLin_YLin_RRogers_linear_IFE,
ZLi_TLin_XWu_Linear_IFE,
TLin_YLin_RRogers_MRyan_Rectangle,
TLin_YLin_XZhang_PPIFE_Elliptic,
TLin_DSheen_XZhang_RQ1_IFE_Elasiticity,
XZhang_PHDThesis}.
The main feature of IFE methods is that they can use interface independent meshes; hence, structured or even Cartesian meshes can be used to solve problems with nontrivial interface geometry (see the right plot in Figure \ref{fig: domain}).
Most of IFE methods are developed for stationary interface problems. There are a few literatures of IFEs on time-dependent interface problems.  For instance, an immersed Eulerian-Lagrangian localized adjoint method was developed to treat transient advection-diffusion equations with interfaces in \cite{Wang_Wang_Yu_Immersed_EL_Interfaces}. In \cite{TLin_DSheen_IFE_Laplace}, IFE methods were applied to parabolic interface problem together with the Laplacian transform. Parabolic problems with moving interfaces were considered in \cite{XHe_TLin_YLin_XZhang_Moving_CNIFE, TLin_YLin_XZhang_MoL_Nonhomo, TLin_YLin_XZhang_IFE_MoL} where Crank-Nicolson-type fully discrete IFE methods and
IFE method of lines were derived through the Galerkin formulation.

For elliptic interface problems, classic IFE methods in Galerkin formulation \cite{ZLi_TLin_YLin_RRogers_linear_IFE,
ZLi_TLin_XWu_Linear_IFE, TLin_YLin_RRogers_MRyan_Rectangle} can usually converge to the exact solution with optimal order in $H^1$ and $L^2$ norm. Recently, the authors in \cite{TLin_YLin_XZhang_PPIFE_Elliptic, XZhang_PHDThesis} observed that their orders of convergence in both $H^1$ and $L^2$ norms can sometimes deteriorate when the mesh size becomes very small, and this order degeneration might be the consequence of the discontinuity of IFE functions across \textit{interface edges} (edges intersected with the interface). Note that IFE functions in \cite{ZLi_TLin_YLin_RRogers_linear_IFE, ZLi_TLin_XWu_Linear_IFE, TLin_YLin_RRogers_MRyan_Rectangle} are constructed so that they are continuous within each interface element. On the boundary of an interface element, the continuity of these IFE functions is only imposed on two endpoints of each edge. This guarantees the continuity of IFE functions on non-interface edges. However, an IFE function is a piecewise polynomial on each interface edge; hence it is usually discontinuous on interface edges. This discontinuity depends on the interface location and the jump of coefficients, and could be large for certain configuration of interface element and diffusion coefficient. When
the mesh is refined, the number of interface elements becomes larger, and such discontinuity over interface edges might be a factor negatively impacting on the global convergence.

To overcome the order degeneration of convergence, a partially penalized immersed finite element (PPIFE) formulation was introduced in \cite{TLin_YLin_XZhang_PPIFE_Elliptic, XZhang_PHDThesis}. In the new formulation, additional stabilization terms generated on interface edges are added to the finite element equations that can penalize the
discontinuity of IFE functions across interface edges. Since the number of interface edges is much smaller than the total number of elements of a Cartesian mesh, the computational cost for generating those partial penalty terms is negligible. For elliptic interface problems, the PPIFE methods can effectively reduce errors around interfaces; hence, maintain the optimal convergence rates under mesh refinement without degeneration.


Our goal here is to develop PPIFE methods for the parabolic interface problem \eqref{eq:parab_eq} - \eqref{eq:parab_eq_jump_2} and to derive the {\it a priori} error estimates for these methods. We present the semi-discrete method and two prototypical fully discrete methods, \emph{i.e.}, the backward Euler method and Crank-Nicolson method in Section 2. In Section 3, the {\it a priori} error estimates are derived for these methods which indicate the optimal convergence from the point of view of polynomials used in the involved IFE subspaces. Finally, numerical examples are provided in Section 4 to validate the theoretical estimates.

In the discussion below, we will use a few general assumptions and notations. First, from now on, we will tacitly assume that the interface problem has a homogeneous boundary condition, \emph{i.e.}, $g = 0$ for the simplicity of presentation. The methods and related analysis can be easily extended to problems with a non-homogeneous boundary
condition through a standard procedure. Second, we will
adopt standard notations and norms of Sobolev spaces. For $r\geq 1$ , we define the
following function spaces:
\[
\tilde H^{r}(\Omega)=\{v: v|_{\Omega^s}\in
H^{r}(\Omega^s), s=+\ \text{or}\ -\}
\]
equipped with the norm
\[
\|v\|^2_{\tilde H^{r}(\Omega)}=\|v\|^2_{H^{r}(\Omega^-)}+\|v\|^2_{H^{r}(\Omega^+)},~~~\forall v\in
\tilde H^{r}(\Omega).
\]
For a function $z(X,t)$ with space variable $X=(x,y)$ and time variable $t$, we consider it as a mapping from the time
interval $[0,T]$ to a normed space $V$ equipped with the norm $\|\cdot\|_{V}$. In particular, for an integer $k\geq 1$, we define
\[
L^k(0,T;V)=\left\{z: [0,T]\rightarrow V ~~\text{measurable, such that}  \int_{0}^{T}\|z(\cdot,t)\|_V^kdt<\infty\right\}
\]
with
\[
\|z\|_{L^k(0,T;V)}=\left(\int_{0}^{T}\|z(\cdot,t)\|_V^kdt\right)^{1/k}.
\]
Also, for $V = \tilde H^{r}(\Omega)$, we will use the
standard function space $H^p(0,T; \tilde H^{r}(\Omega))$ for $p \geq 0, r \geq 1$.

\par
In addition, we will use $C$ with or without subscript to denote a generic positive constant which may have different
values according to its occurrence. For simplicity, we will use $u_t$, $u_{tt}$, etc., to denote the partial derivatives
of a function $u$ with
respect to the time variable $t$.
%
\par

\section{Partially Penalized Immersed Finite Element Methods}
\setcounter{section}{2}\setcounter{equation}{0}
In this section, we first derive a weak formulation of the parabolic interface problem \eqref{eq:parab_eq} - \eqref{eq:parab_eq_jump_2} based on Cartesian meshes. Then we recall bilinear IFE functions and spaces defined on rectangular meshes from \cite{XHe_TLin_YLin_Bilinear_Approximation, TLin_YLin_RRogers_MRyan_Rectangle}. The construction of linear IFE functions on triangular meshes is similar, so we refer to \cite{ZLi_TLin_YLin_RRogers_linear_IFE, ZLi_TLin_XWu_Linear_IFE} for more details. Finally, we introduce the partially penalized immersed finite element methods for the parabolic interface problem.

\subsection{Weak Form on Continuous Level}
Let $\mathcal{T}_h$ be a Cartesian (either triangular or rectangular) mesh consisting of elements whose diameters are
not larger than $h$. We denote by $\mathcal
{N}_h$ and $\mathcal{E}_h$ the set of all vertices and edges in $\mathcal{T}_h$, respectively. The set of all interior edges are denoted by $\mathcal {\mathring{E}}_h$. If an element is cut by the interface $\Gamma$, we call it an interface
element; otherwise, it is called a non-interface element.
Let $\mathcal{T}_h^i$ be
the set of interface elements and
$\mathcal{T}_h^n$ be the set of non-interface elements. Similarly, we define the set of interface edges and the set of
non-interface edges which are denoted by $\mathcal{E}_h^i$ and $\mathcal{E}_h^n$, respectively.
Also, we use $\mathcal {\mathring{E}}_h^i$ and
$\mathcal{\mathring{E}}_h^n$ to denote the set of interior interface
edges and interior non-interface edges, respectively. With the assumption that $\Gamma \subset \Omega$ we have
$\mathcal {\mathring{E}}_h^i = \mathcal{E}_h^i$.

We assign a unit normal vector $\mathbf{n}_B$ to every edge $B\in {{\mathcal{E}}}_h$. If $B$ is an interior edge, we let $K_{B,1}$ and $K_{B,2}$ be the two elements that share the common edge $B$ and we assume that the normal vector $\mathbf{n}_B$ is oriented
from $K_{B,1}$ to $K_{B,2}$. For a function $u$ defined on $K_{B,1}\cup K_{B,2}$, we set its average and jump on $B$ as follows
\begin{equation*}
\aver{u}_B=\frac{1}{2}\left((u|_{K_{B,1}})|_B+(u|_{K_{B,2}})|_B\right),~~~~
\jump{u}_B=(u|_{K_{B,1}})|_B-(u|_{K_{B,2}})|_B.
\end{equation*}
 If $B$ is on the boundary $\partial\Omega$,
$\mathbf{n}_B$ is taken to be the unit outward vector normal to
$\partial\Omega$, and we let
\[
\aver{u}_B=\jump{u}_B=u|_{B}.
\]
For simplicity, we often drop the subscript $B$ from these notations if there is no danger to cause any confusions.

Without loss of generality, we assume that the interface $\Gamma$ intersects with the edge of each interface element $K \in \mathcal{T}_h^i$ at two points. We then partition
$K$ into two sub-elements $K^-$ and $K^+$ by the line segment connecting these two interface points, see the illustration given in Figure \ref{fig: interface element 4pc}.

To describe a weak form for the parabolic interface problem, we introduce the following space
\begin{equation}\label{eq: sobolev space}
  V_h = \{v\in L^2(\Omega): v \text{ satisfies conditions} \textbf{ (HV1) - (HV4)}\}
\end{equation}
\begin{description}
  \item[(HV1)] $v|_K\in H^1(K), v|_{K^s}\in H^2(K^s), s = \pm, \forall K\in \mathcal{T}_h$.
  \item[(HV2)] $v$ is continuous at every $X\in \mathcal{N}_h$.
  \item[(HV3)] $v$ is continuous across each $B\in {\mathring{\mathcal{E}}}_h^n$.
  \item[(HV4)] $v|_{\partial \Omega} = 0$.
\end{description}
\begin{figure}
  \centering
  \includegraphics[width=.45\textwidth]{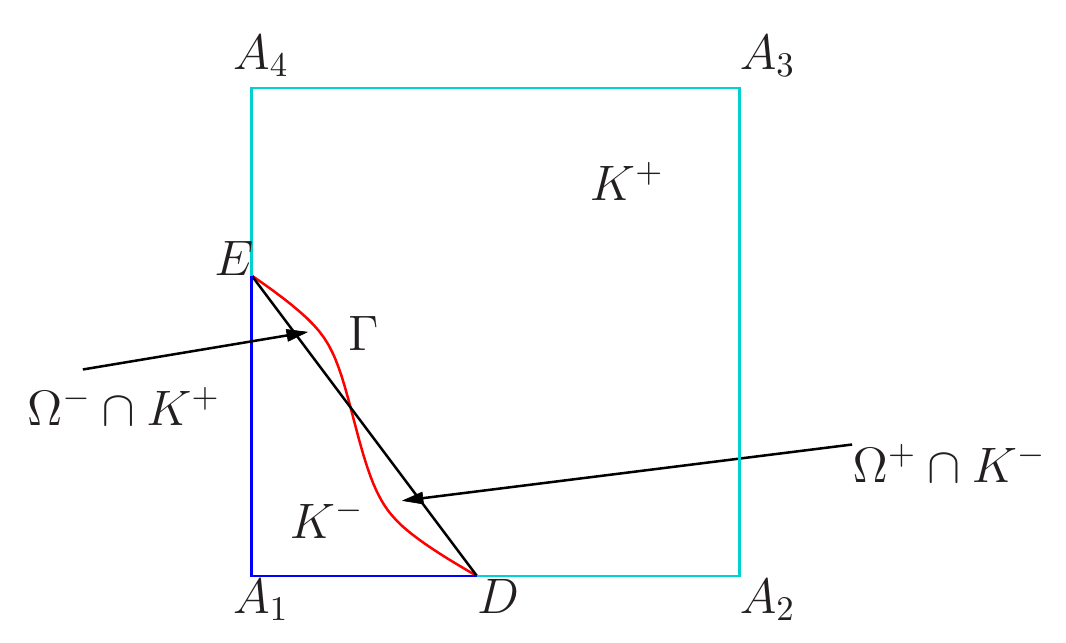}\\
  \caption{An interface element}
  \label{fig: interface element 4pc}
\end{figure}
Note that functions in $V_h$ are allowed to be discontinuous on interface edges.
We now derive a weak form with the space $V_h$ for the parabolic interface problem
\eqref{eq:parab_eq} - \eqref{eq:parab_eq_jump_2}. First, we assume that its exact solution $u$ is in
$\tilde H^{2}(\Omega)$. Then, we multiply equation \eqref{eq:parab_eq} by a test function $v\in V_h$ and integrate both sides on each element
$K\in\mathcal{T}_h$. If $K$ is a non-interface element, a direct application of
 Green's formula leads to
\begin{equation} \label{weak_local}
\int_Ku_t vdX+\int_{K}\beta\nabla
u\cdot\nabla vdX-\int_{\partial K}\beta\nabla
u\cdot{\mathbf{n}_K}vds=\int_Kf v dX.
\end{equation}
If $K$ is an interface element, we assume that the interface $\Gamma$ intersects $\partial K$ at points $D$ and $E$. Then, without loss of generality, we assume that
$\Gamma$ and the line $\overline{DE}$ divide $K$ into up to four sub-elements, see the illustration in Figure \ref{fig: interface element 4pc} for a rectangle interface element, such that
\[
K=(\Omega^+\cap K^+)\cup(\Omega^-\cap K^-)\cup(\Omega^+\cap K^-)\cup(\Omega^-\cap K^+).
\]
Now, applying Green's formula separately on these four sub-elements, we get
\begin{eqnarray}
&& -\int_{K}\nabla\cdot(\beta\nabla u)vdX \nonumber\\
&=& -\int_{\Omega^+\cap K^+}\nabla\cdot(\beta^+\nabla
u)vdX-\int_{\Omega^-\cap K^-}\nabla\cdot(\beta^-\nabla u)vdX \nonumber\\
&& -\int_{\Omega^+\cap K^-}\nabla\cdot(\beta^+\nabla
u)vdX-\int_{\Omega^-\cap K^+}\nabla\cdot(\beta^-\nabla u)vdX \nonumber\\
&=&\int_{\Omega^+\cap K^+}\beta^+\nabla u\cdot\nabla vdX-\int_{\partial
(\Omega^+\cap K^+)}\beta^+\nabla u\cdot{\mathbf{n}}vds+\int_{\Omega^-\cap K^-}\beta^-\nabla u\cdot\nabla vdX-\int_{\partial
(\Omega^-\cap K^-)}\beta^-\nabla u\cdot{\mathbf{n}}vds \nonumber\\
&& +\int_{\Omega^+\cap K^-}\beta^+\nabla
u\cdot\nabla vdX -\int_{\partial(\Omega^+\cap K^-)}\beta^+\nabla u\cdot{\mathbf{n}}vds +\int_{\Omega^-\cap K^+}\beta^-\nabla
u\cdot\nabla vdX-\int_{\partial(\Omega^-\cap K^+)}\beta^-\nabla u\cdot{\mathbf{n}}vds \nonumber\\
&=&\int_{K}\beta\nabla u\cdot\nabla vdX-\int_{\partial K}\beta\nabla
u\cdot{\mathbf{n}}vds - \int_{K\cap\Gamma}\jump{\beta\nabla
u\cdot{\mathbf{n}}}_{K\cap\Gamma}vds \nonumber\\
&=&\int_{K}\beta\nabla u\cdot\nabla vdX-\int_{\partial K}\beta\nabla
u\cdot{\mathbf{n}}vds. \label{eq: interface Green formual}
\end{eqnarray}
The last equality is due to the interface jump condition \eqref{eq:parab_eq_jump_2}. The derivation of \eqref{eq: interface Green formual} implies that \eqref{weak_local} also holds on interface elements.
\begin{remark}
Figure \ref{fig: interface element 4pc} is a typical configuration of an interface element. If the interface is smooth enough and the mesh size is sufficiently small, an interface is usually divided into three sub-elements, i.e., one of the two terms $\Omega^+\cap K^-$ and $\Omega^-\cap K^+$ is an empty set. In this case, the related discussion is similar but slightly simpler.
\end{remark}
Summarizing \eqref{weak_local} over all elements indicates
 \begin{equation}
\int_\Omega u_tvdX+\sum\limits_{K\in\mathcal{T}_h}\int_K\beta\nabla u\cdot \nabla
vdX-\sum\limits_{B\in \mathring{\mathcal{E}}_h^i}\int_{B}\aver{\beta\nabla u\cdot
\mathbf{n}_B}\jump{v}ds=\int_{\Omega}fvdX.
\end{equation}
Let $H_h=\tilde H^2(\Omega)+V_h$ on which we introduce a bilinear form $a_\epsilon$:
$H_h\times H_h\rightarrow \mathbb{R}$:
\begin{eqnarray}\nonumber
a_\epsilon(w,v)&=&\sum\limits_{K\in\mathcal{T}_h}\int_K\beta\nabla
v\cdot \nabla w dX-\sum\limits_{B\in \mathring{\mathcal{E}}_h^i}\int_{B}\aver{\beta\nabla w\cdot
\mathbf{n}_B}\jump{v}ds\\
&&+\epsilon\sum\limits_{B\in \mathring{\mathcal{E}}_h^i}\int_{B}\aver{\beta\nabla v\cdot
\mathbf{n}_B}\jump{w}ds+\sum\limits_{B\in
\mathring{\mathcal{E}}_h^i}\int_{B}\frac{\sigma^0_B}{|B|^{\alpha}}\jump{v}\jump{w}ds,\label{eq: bilinear form}
\end{eqnarray}
where $\alpha>0$, $\sigma^0_B\geq 0$ and $|B|$ means the length of
$B$. Note that the regularity of $u$ leads to
\[
\epsilon\sum\limits_{B\in \mathcal{\mathring E}_h^i}\int_{B}\aver{\beta\nabla v\cdot
\mathbf{n}_B}\jump{u}ds=0,
\ \ \sum\limits_{B\in
\mathcal{\mathring E}_h^i}\int_{B}\frac{\sigma^0_B}{|B|^{\alpha}}\jump{v}\jump{u}ds=0.
\]
We also define the following linear form
\[
L(v)=\int_{\Omega}fvdX.
\]
Finally, we have the following weak form of the parabolic interface problem
\eqref{eq:parab_eq}-\eqref{eq:parab_eq_jump_2}: find $u: [0,T] \rightarrow \tilde H^{2}(\Omega)$ that satisfies
\eqref{eq:parab_eq_jump_1}, \eqref{eq:parab_eq_jump_2}, and
\begin{eqnarray}
\left(u_t,
v\right)+a_\epsilon(u,v)&=&L(v), \label{eq:weak_form}\
\ \forall v\in V_h,\\
u(X,0)&=&u_0(X), \ \ \forall X \in\Omega. \label{eq:weak_init_cond}
\end{eqnarray}

\subsection{Immersed Finite Element Functions}
In this subsection, to be self-contained, we recall IFE spaces that approximate $V_h$. We describe the bilinear IFE space with a little more details, and refer readers to
\cite{ZLi_TLin_YLin_RRogers_linear_IFE, ZLi_TLin_XWu_Linear_IFE} for corresponding descriptions of the linear IFE space on a triangular Cartesian mesh. Since an IFE space
uses standard finite element functions on each non-interface element, we will focus on the presentation of IFE functions on interface elements.

The bilinear ($Q_1$) immersed finite element functions were introduced in \cite{XHe_TLin_YLin_Bilinear_Approximation, TLin_YLin_RRogers_MRyan_Rectangle}. On each interface element, a local IFE space uses IFE functions in the form of piecewise bilinear polynomials constructed according to interface jump conditions.
Specifically, we partition each interface element $K = \square A_1A_2A_3A_4$ into two sub-elements $K^-$ and $K^+$ by the line connecting points
$D$ and $E$ where the interface $\Gamma$ intersects with $\partial K$, see Figure \ref{fig: interface elements} for illustrations. Then we construct
four bilinear IFE shape functions $\phi_i, i = 1, 2, 3, 4$ associated with the vertices of $K$ such that


\begin{figure}[ht]
  \centering
  \includegraphics[width=0.25\textwidth]{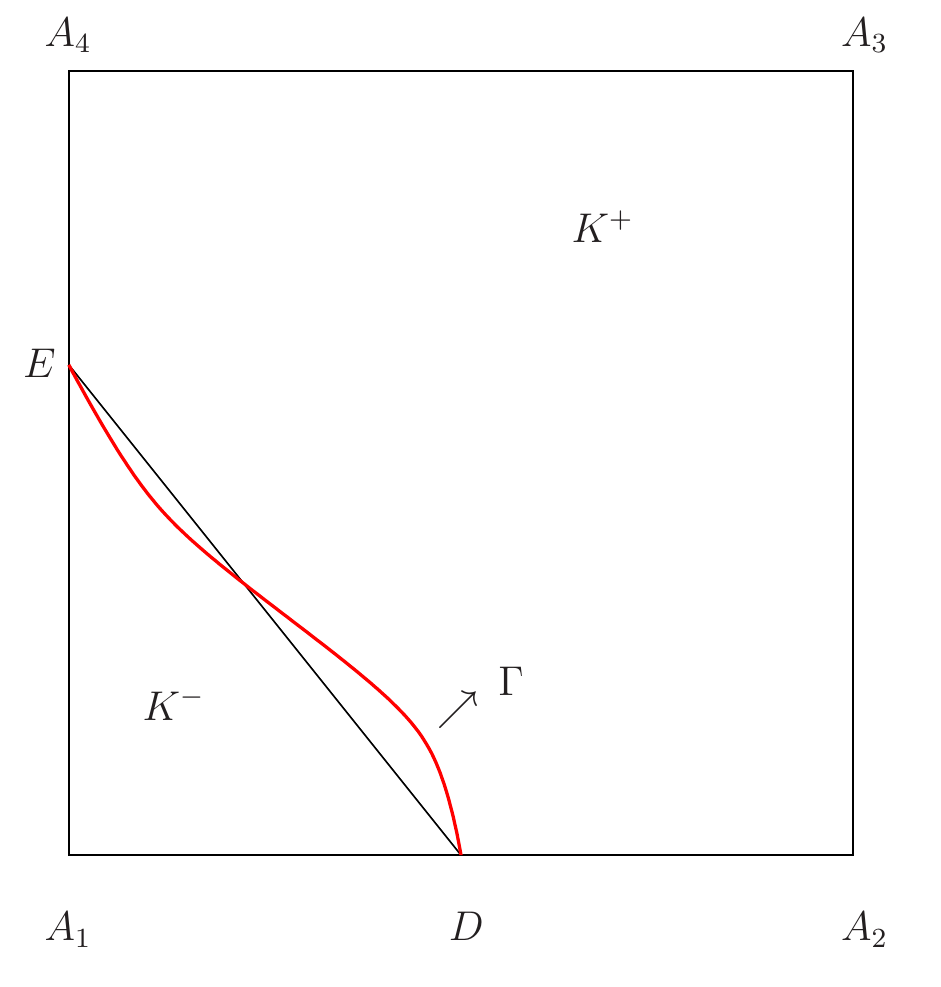}~~~~
  \includegraphics[width=0.25\textwidth]{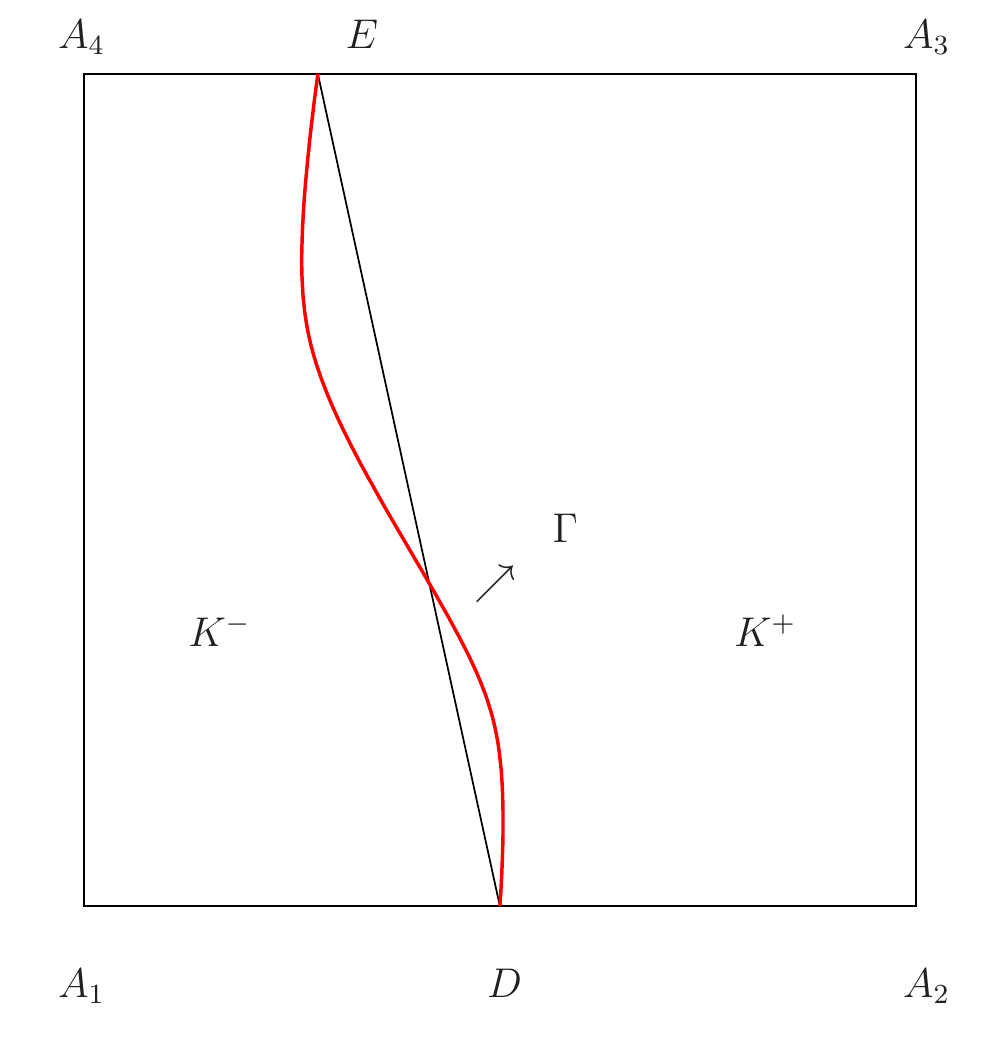}\\
  \caption{Type I and Type II interface rectangles}
  \label{fig: interface elements}
\end{figure}
\begin{equation}\label{eq: bilinear IFE function}
    \phi_{i}(x,y) =
    \left\{
      \begin{array}{cc}
        \phi_{i}^{+}(x,y) = a_i^+ + b_i^+x + c_i^+y + d_i^+xy,~~~~  &\text{if}~ (x,y)\in K^+, \vspace{1mm}\\
        \phi_{i}^{-}(x,y) = a_i^- + b_i^-x + c_i^-y + d_i^-xy,~~~~  &\text{if}~ (x,y)\in K^-, \vspace{1mm}\\
      \end{array}
    \right.
\end{equation}
according to the following constraints:
\begin{itemize}
  \item nodal value condition:
  \begin{equation}\label{eq: nodal condition}
    \phi_{i}(A_j) = \delta_{ij},~~~~~i,j = 1,2,3,4.
\end{equation}
  \item continuity on $\overline{DE}$
  \begin{equation}\label{eq: continuity1}
    \jump{\phi_{i}{(D)}} =0, ~~~~\jump{\phi_{i}{(E)}} = 0,~~~~\bigjump{\frac{\partial^2 \phi_{i}}{\partial x\partial y}} = 0.
\end{equation}
  \item continuity of normal component of flux
  \begin{equation}\label{eq: continuity2}
    \int_{\overline{DE}}\bigjump{\beta\frac{\partial \phi_{i}}{\partial n}} \mbox{d}s= 0.
  \end{equation}
\end{itemize}
It has been shown \cite{XHe_Thesis_Bilinear_IFE,XHe_TLin_YLin_Bilinear_Approximation} that conditions specified in \eqref{eq: nodal condition} - \eqref{eq: continuity2} can uniquely determine these shape functions. Figure \ref{fig: IFE 2D Bilinear Local_Basis} provides a comparison of FE and IFE shape functions.
\begin{figure}[ht]
  \centering
  \includegraphics[width=0.3\textwidth]{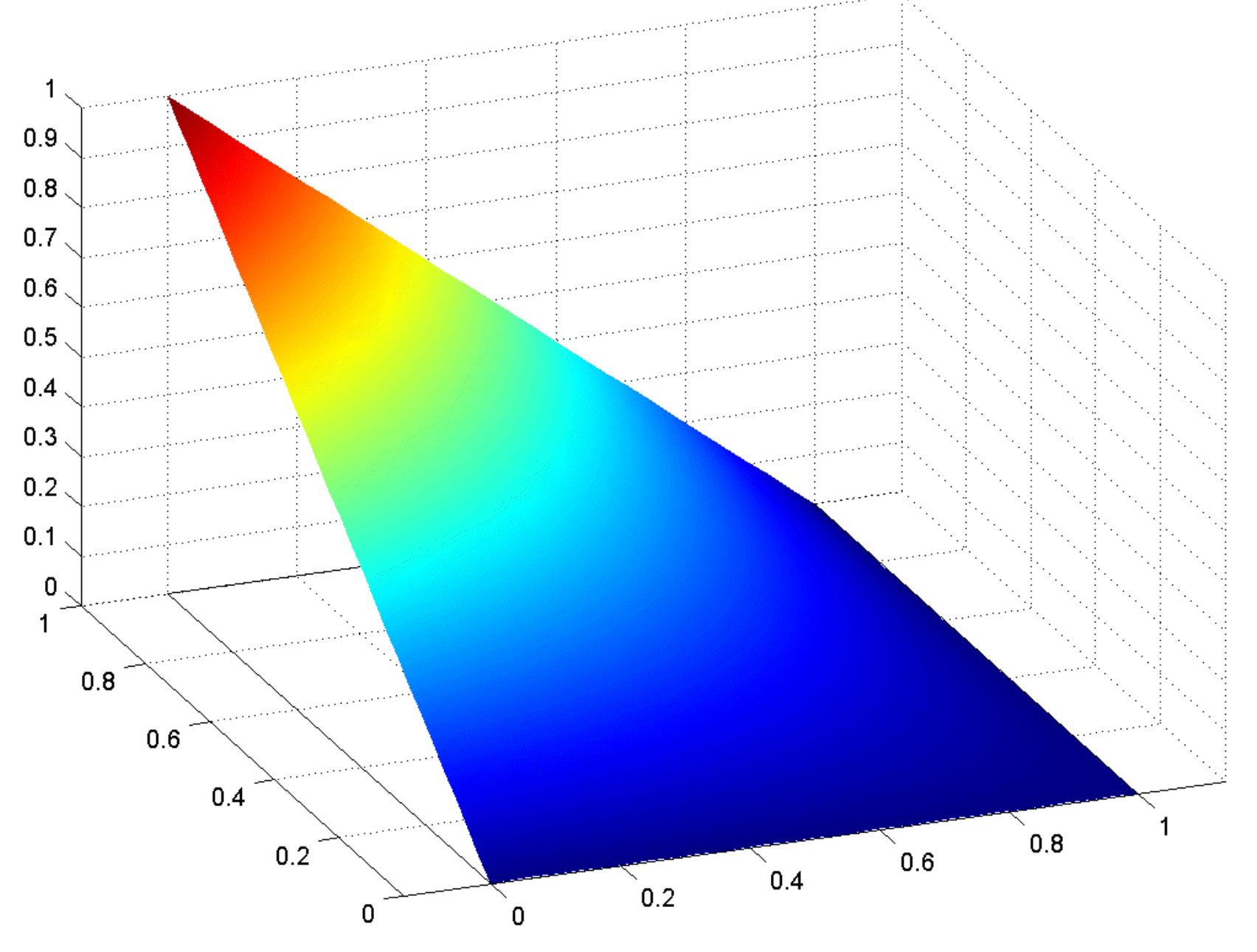}~~
  \includegraphics[width=0.3\textwidth]{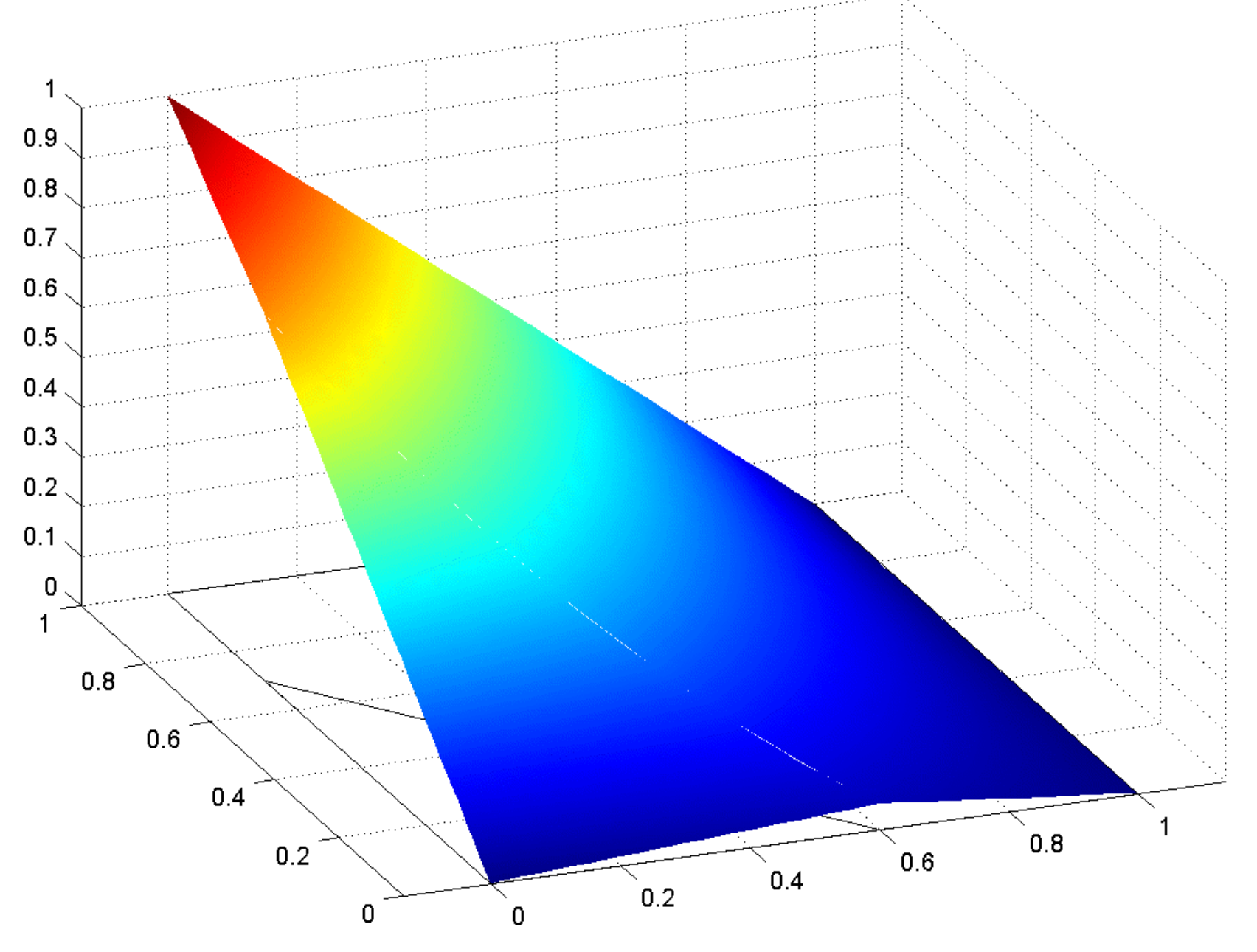}~~
  \includegraphics[width=0.3\textwidth]{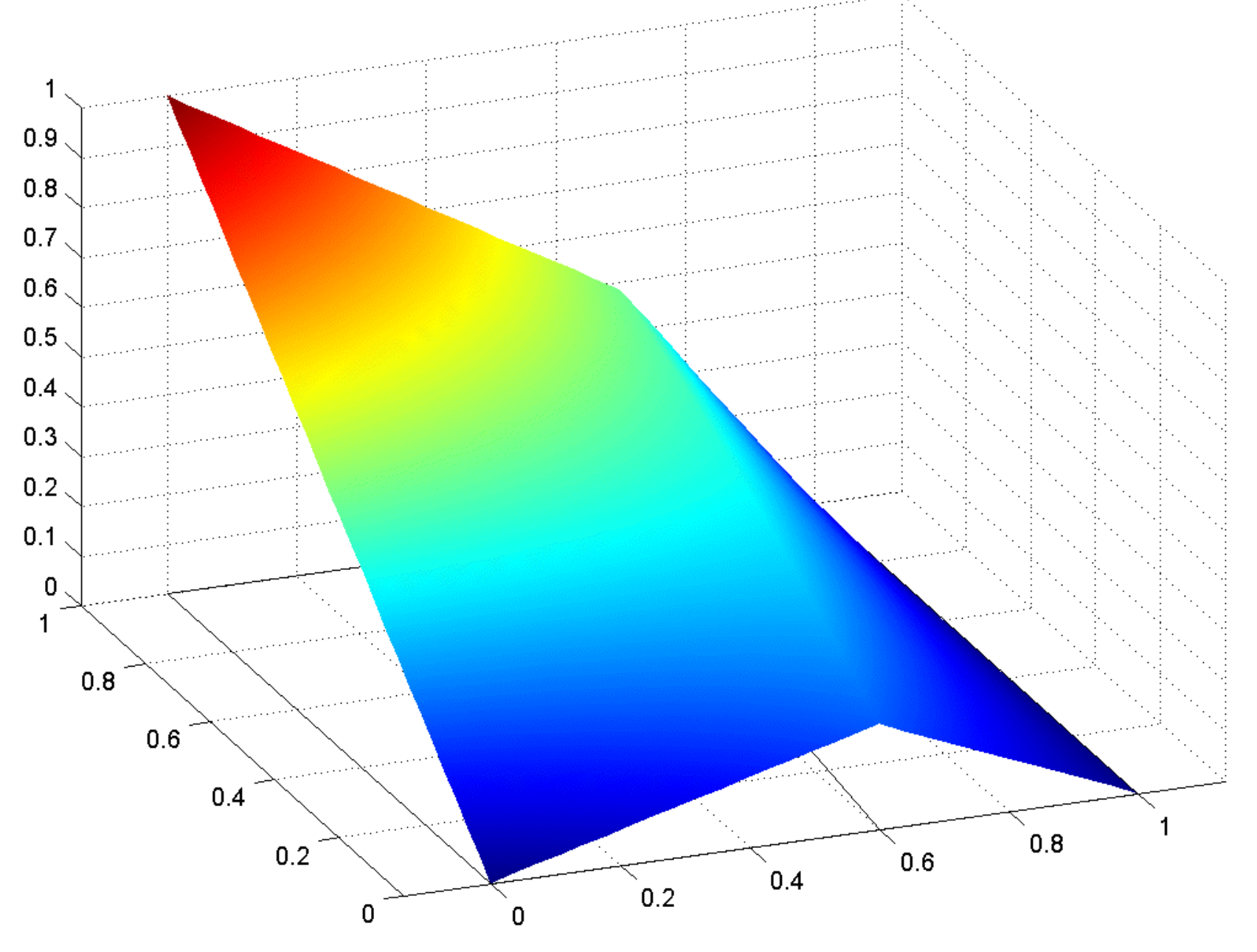}\\
  \caption{Bilinear FE/IFE local basis functions. From left: FE, IFE (Type I), IFE (Type II)}
  \label{fig: IFE 2D Bilinear Local_Basis}
\end{figure}

Similarly, see more details in \cite{ZLi_TLin_YLin_RRogers_linear_IFE, ZLi_TLin_XWu_Linear_IFE}, on a triangular interface element $K = \triangle A_1A_2A_3$, we can construct three linear IFE shape functions $\phi_{i}, i = 1, 2, 3$ that
satisfy the first two equations in \eqref{eq: continuity1}, \eqref{eq: continuity2}, and
\begin{eqnarray*}
\phi_{i}(A_j) = \delta_{ij},~~~~~i,j = 1,2,3.
\end{eqnarray*}

These IFE shape functions possess a few notable properties such as their consistence with the corresponding standard Lagrange type FE shape functions and their formation of partition of unity. We refer readers to \cite{XHe_Thesis_Bilinear_IFE,XHe_TLin_YLin_Bilinear_Approximation, ZLi_TLin_YLin_RRogers_linear_IFE,XZhang_PHDThesis} for more details. \\

Then, on each element $K \in \mathcal{T}_h$, we define the local IFE space as follows:
\begin{eqnarray*}
S_h(K) = span\{\phi_i,~~1 \leq i \leq d_K\},~~d_K = \begin{cases}
3, & \text{if $K$ is a triangular element}, \\
4, &\text{if $K$ is a rectangular element},
\end{cases}
\end{eqnarray*}
where $\phi_i, 1 \leq i \leq d_K$ are the standard linear or bilinear
Lagrange type FE shape functions for $K \in \mathcal{T}_h^n$; otherwise, they are the IFE shape functions described above.
Finally, the IFE spaces on the whole solution domain $\Omega$ are defined as follows:
\begin{eqnarray*}
S_h(\Omega) &=& \{v \in V_h~:~v|_K \in S_h(K),~\forall K \in \mathcal{T}_h\}.
\end{eqnarray*}

\begin{figure}[ht]
  \centering
  \includegraphics[width=0.4\textwidth]{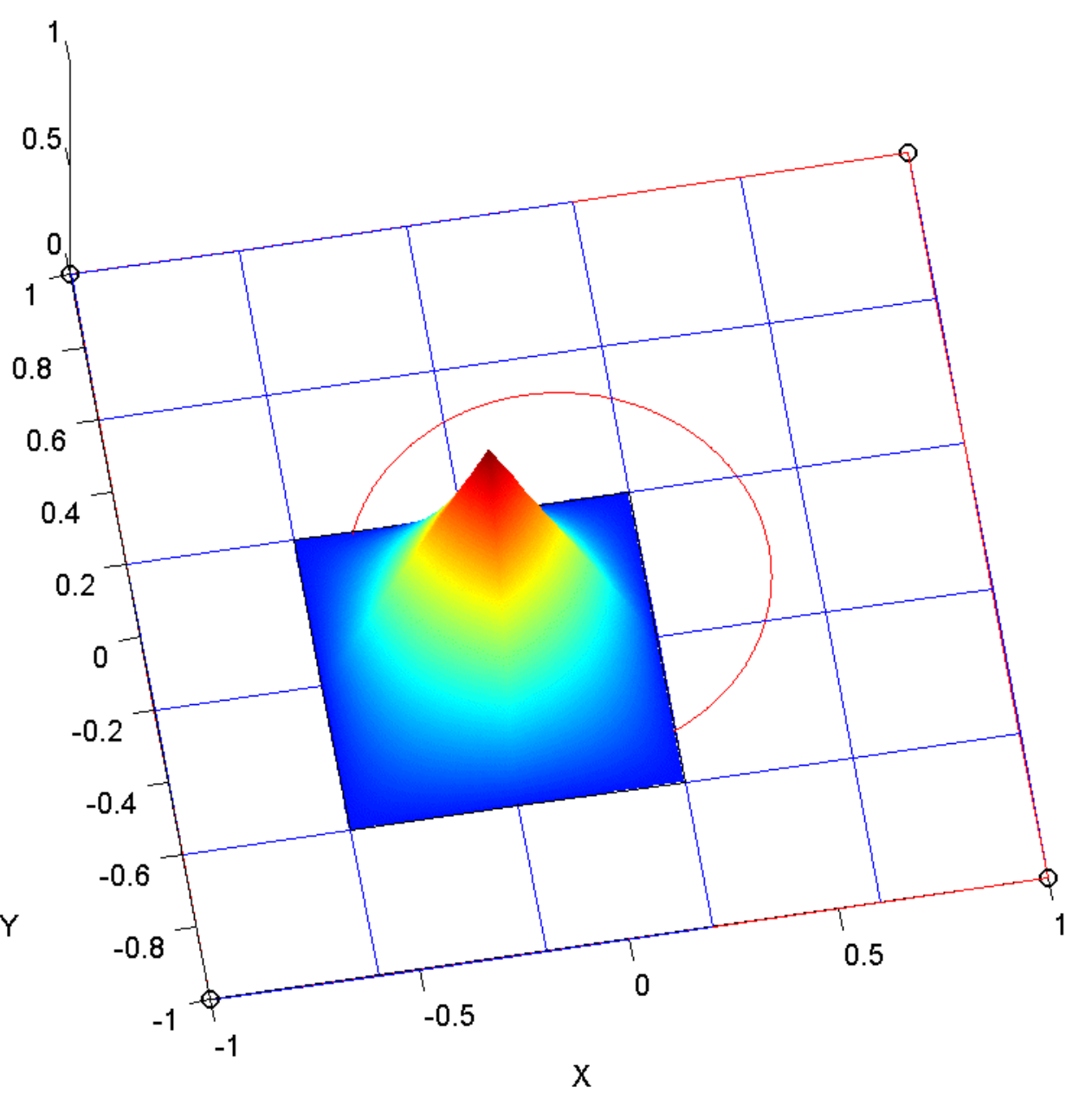}~~~
  \includegraphics[width=0.4\textwidth]{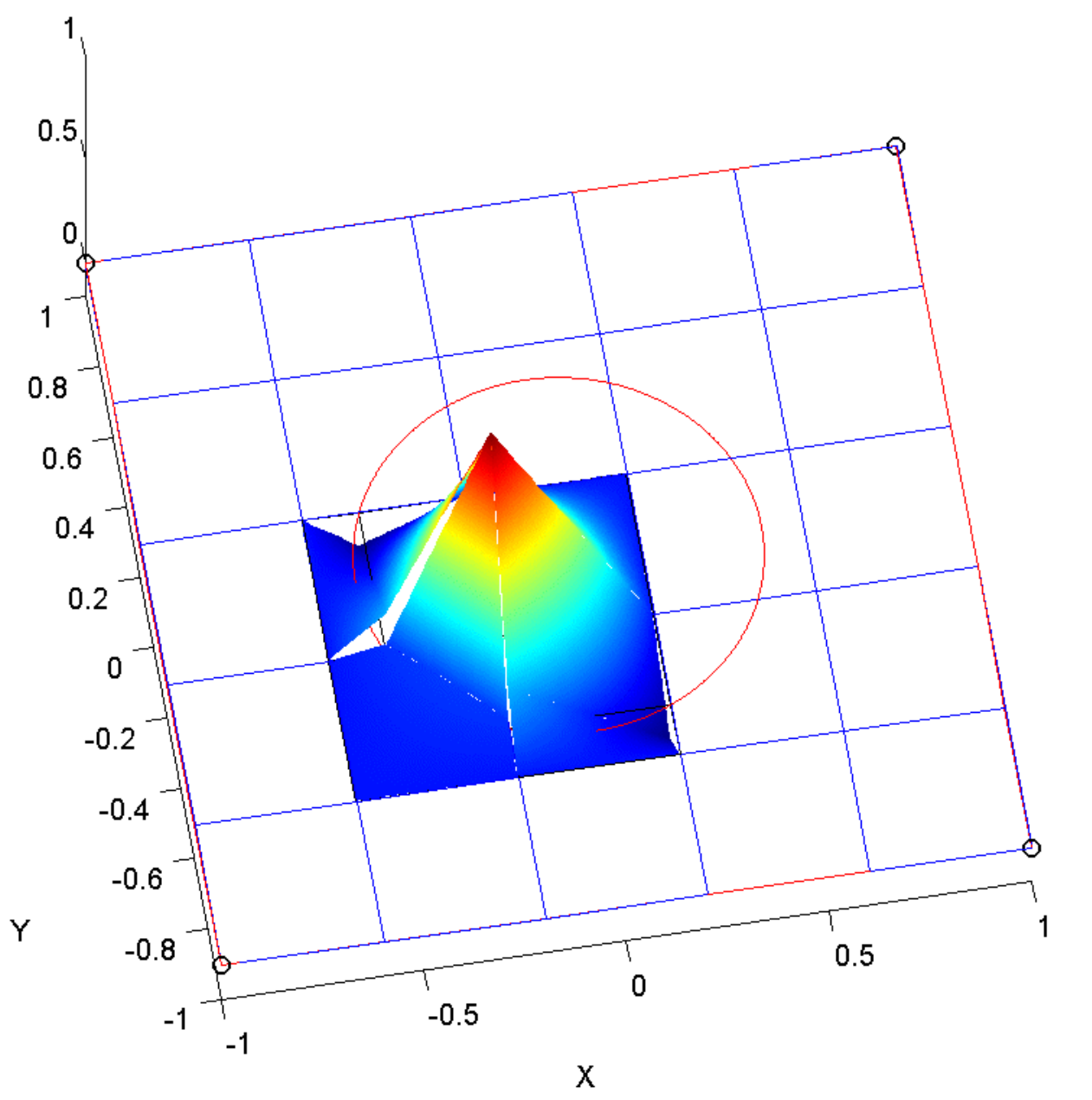}\\
  \caption{Bilinear FE (left) and IFE (right) global basis functions}
  \label{fig: IFE 2D Bilinear Global_Basis}
\end{figure}

\begin{remark}
We note that an IFE function may not be continuous across the element boundary that intersects with the interface. An IFE shape function is usually not zero on an interface edge, see the values on the edge between the points $(0,0)$ and $(0,1)$ for the two IFE shape functions plotted in Figure \ref{fig: IFE 2D Bilinear Local_Basis}. On this interface edge, the shape functions vanish at two endpoints, but not on the entire edge. The maximum of the absolute values of the shape on that edge is determined by the geometrical and material configuration on an interface element. When the local IFE shape functions are put together to form a Lagrange type global IFE basis function associated with a node in a mesh, it is inevitably to be discontinuous on interface edges in elements around that node, as illustrated by the cracks in a global IFE basis function plotted in Figure \ref{fig: IFE 2D Bilinear Global_Basis}. As observed in \cite{TLin_YLin_XZhang_PPIFE_Elliptic, XZhang_PHDThesis}, this discontinuity on interface edges might be a factor causing the deterioration of the convergence of classic IFE solution around the interface, and this motivates us to add partial penalty on interface edges for alleviating this adversary.
\end{remark}


\subsection{Partially Penalized Immersed Finite Element Methods}
In this subsection we use the global IFE space $S_h(\Omega)$ to discretize the weak form \eqref{eq:weak_form} and \eqref{eq:weak_init_cond} for the parabolic interface problem. While the standard
semi-discrete or many fully discrete frameworks can be applied, we will focus on the following prototypical schemes because of their popularity.  \\


\noindent
{\bf A semi-discrete PPIFE method}: Find $u_h: [0,T]\rightarrow
S_{h}(\Omega)$ such that 
\begin{eqnarray}
\left( u_{h,t},
v_h\right)+a_\epsilon(u_h,v_h)&=&L(v_h),\ \ \forall v_h\in
S_{h}(\Omega), \label{eq:DG-IFE_semi}\\
u_h(X,0)&=&\tilde u_h(X),\ \ \forall X\in\Omega, \label{eq:DG-IFE_semi_ic}
\end{eqnarray}
where $\tilde u_h$ is an approximation of $u_0$ in the space
$S_{h}(\Omega)$. According to the analysis to be carried out in the next section, $\tilde u_h$ can be chosen as the interpolation of $u_0$ or the elliptic projection of
$u_0$ in the IFE space $S_h(\Omega)$. \\

\noindent
{\bf A fully discrete PPIFE method}: For a positive integer $N_t$, we let $\Delta t=T/N_t$ which is the time step and
let $t^n=n\Delta t$ for integer $n \geq 0$. Also, for a sequence $\varphi^n, n \geq 1$, we let
\[
\partial_t \varphi^n=\frac{\varphi^{n}-\varphi^{n-1}}{\Delta t}.
\]
Then, the fully discrete PPIFE method is to find a sequence $\big\{u_h^n\big\}_{n=1}^{N_t}$ of
functions in $S_{h}(\Omega)$ such that 
\begin{eqnarray}
\left(\partial_t u_h^n,
v_h\right)+a_\epsilon(\theta
u_h^{n}+(1-\theta)u_h^{n-1},v_h)&=&\theta L^n(v_h)+(1-\theta)L^{n-1}(v_h),~~ \forall v_h\in
S_{h}(\Omega), \label{eq:DG-IFE_full_disc}\\
u_h^0(X)&=&\tilde u_h(X),\ \ \forall X\in\Omega. \label{eq:DG-IFE_full_disc_ic}
\end{eqnarray}
Here, $L^n(v_h) = \int_\Omega f(X, t^n) v_h(X) dX, n \geq 0$ and $\theta$ is a parameter chosen from $[0,1]$.
Popular choices for $\theta$ are $\theta = 0, \theta = 1$ and $\theta = 1/2$ representing the forward Euler method, the backward Euler method, and
the Crank-Nicolson method, respectively.

\begin{remark}
The bilinear form $a_\epsilon(\cdot, \cdot)$ in \eqref{eq:weak_form} is almost the same as that used in the interior penalty DG finite element
methods for the standard elliptic boundary value problem \cite{ZChen_FEM, Hesthaven_Warburton_Nodal_DG, Riviere_DG_book} except that it contains integrals over
interface edges instead of all the edges. This is why we call IFE methods based on this bilinear form partially penalized IFE (PPIFE) methods. As suggested by DG finite element methods, the parameter $\epsilon$ in this bilinear form is usually chosen as $-1$, $0$, or $1$. Note that $a_\epsilon(\cdot,\cdot)$ is symmetric if $\epsilon=-1$ and is nonsymmetric otherwise.
\end{remark}

\section{Error Estimations for PPIFE Methods}
\setcounter{section}{3}\setcounter{equation}{0}
The goal of this section is to derive the {\it a priori} error estimates for the PPIFE methods developed in the previous section. As usual, without loss of generality for error estimation, we assume that $g(X, t) = 0$ in the boundary condition \eqref{eq:parab_eq_bc} and assume $\Gamma \cap \partial \Omega = \emptyset$.
The error bounds will be given in an energy norm
that is equivalent to the standard semi-$H^1$ norm. These error bounds show that these PPIFE methods converge optimally with respect to the polynomials employed.

\subsection{Some Preliminary Estimates}

\noindent

First, for every $v\in V_h$, we define its energy norm as follows:
\[
\|v\|_h=\left(\sum\limits_{K\in\mathcal{T}_h}\int_K\beta\nabla
v\cdot \nabla vdX+\sum\limits_{B\in \mathring{\mathcal{
E}}_h^i}\int_{B}\frac{\sigma_B^0}{|B|^{\alpha}}\jump{v}\jump{v}ds\right)^{1/2}.
\]
For an element $K\in
\mathcal{T}_h$, let $|K|$ denote the area of $K$. It is well known that the following trace
inequalities \cite{Riviere_DG_book} hold:

\begin{lemma} There exists a constant $C$ independent of $h$ such that for every $K\in \mathcal{T}_h$,
\begin{eqnarray}
&&\|v\|_{L^2(B)}\leq C|B|^{1/2}|K|^{-1/2}(\|v\|_{L^2(K)}+h\|\nabla
v\|_{L^2(K)}),\ \forall v\in H^1(K), ~B\subset\partial K, \label{eq:trace_1}\\
&&\|\nabla v\|_{L^2(B)}\leq C|B|^{1/2}|K|^{-1/2}(\|\nabla
v\|_{L^2(K)}+h\|\nabla^2 v\|_{L^2(K)}),\ \forall v\in H^2(K), ~B\subset\partial K. \label{eq:trace_2}
\end{eqnarray}
\end{lemma}
Since the local IFE space $S_h(K)\subset H^1(K)$ for all $K\in
\mathcal{T}_h$ (e.g. \cite{XHe_Thesis_Bilinear_IFE,XHe_TLin_YLin_Bilinear_Approximation,ZLi_TLin_YLin_RRogers_linear_IFE}), the trace inequality \eqref{eq:trace_1} is valid for all $v\in S_h(K)$. However, for $K\in
\mathcal{T}_h^i$, a function $v\in S_h(K)$  does not belong to $H^2(K)$ in general. So the second
trace inequality \eqref{eq:trace_2} cannot be directly applied to IFE functions. Nevertheless, for linear
and bilinear IFE functions, the corresponding  trace inequalities
have been established in \cite{TLin_YLin_XZhang_PPIFE_Elliptic}. The related results are summarized in the following lemma.

\begin{lemma}
There exists a constant $C$ independent of interface location and $h$ but depending on the ratio of coefficients $\beta^+$ and $\beta^-$ such that for every
linear or bilinear IFE function $v$ on $K\in \mathcal{T}_h^i$,
\begin{eqnarray}
&&\|\beta v_d\|_{L^2(B)}\leq Ch^{1/2}|K|^{-1/2}\|\sqrt{\beta}\nabla
v\|_{L^2(K)},\ \forall v\in S_h(K), B\subset\partial K, d=x \
\text{or}\ y, \label{eq:trace_ife_1}\\
&&\|\beta \nabla v\cdot \mathbf{n}_B\|_{L^2(B)}\leq
Ch^{1/2}|K|^{-1/2}\|\sqrt{\beta}\nabla v\|_{L^2(K)},\ \ \forall v\in
 S_h(K), B\subset\partial K. \label{eq:trace_ife_2}
\end{eqnarray}
\end{lemma}
As in \cite{TLin_YLin_XZhang_PPIFE_Elliptic}, using Young's inequality, trace inequalities and the definition of $\|\cdot\|_h$,
we can prove the coercivity of the bilinear form $a_\epsilon(\cdot,\cdot)$
on the IFE space $S_{h}(\Omega)$ with
respect to the energy norm $\|\cdot\|_h$. The result is stated in the lemma below.

\begin{lemma}  There exists a constant $\kappa>0$ such that
 \begin{equation}
 a_\epsilon(v_h,v_h)\geq\kappa\|v_h\|_h^2,\ \ \forall v_h\in
 S_{h}(\Omega) \label{eq:coarcivity}
 \end{equation}
 holds for $\epsilon=1$ unconditionally and holds for
 $\epsilon=0$ or $\epsilon=-1$ when the penalty
 parameter $\sigma_B^0$ in $a_\epsilon(\cdot,\cdot)$ is large
 enough and $\alpha\geq 1$.
\end{lemma}

For any $t\in[0,T]$, we define the elliptic projection of the exact solution $u(\cdot, t)$ as the IFE function $\tilde u_h(\cdot, t) \in S_h(\Omega)$ by
\begin{equation}
a_\epsilon(u-\tilde u_h, v_h)=0,\ \ \forall v_h\in
S_{h}(\Omega). \label{eq:ellip_proj}
\end{equation}

\begin{lemma}
Assume the exact solution $u$ is in $H^2(0,T;\tilde H^3(\Omega))$ and $\alpha=1$. Then there exists a constant $C$ such that for every $t\in [0,T]$ the following error estimates hold
\begin{eqnarray}
  &&\|u-\tilde u_h\|_h\leq Ch\|u\|_{\tilde H^3(\Omega)}, \label{eq:ell_proj_est_1} \\
  &&\|(u-\tilde u_h)_{t}\|_h\leq Ch\|u_{t}\|_{\tilde H^3(\Omega)}. \label{eq:ell_proj_est_2} \\
  &&\|(u-\tilde u_h)_{tt}\|_h\leq Ch\|u_{tt}\|_{\tilde H^3(\Omega)}. \label{eq:ell_proj_est_3}
\end{eqnarray}
\end{lemma}
\begin{proof}
First, the estimate \eqref{eq:ell_proj_est_1} follows directly from the estimate derived for the PPIFE methods for elliptic problems in \cite{TLin_YLin_XZhang_PPIFE_Elliptic}.
Because of the linearity of the bilinear form, we have
that
\[
a_\epsilon\left((u-\tilde
u_h)_t, v_h\right)=\frac{d}{dt}a_\epsilon(u-\tilde u_h, v_h)=0,\ \
\forall v_h\in S_{h}(\Omega).
\]
This indicates that the time derivative of the elliptic projection is
the elliptic projection of the time derivative. Thus, for any given $t\in [0,T]$, $u_t\in\tilde H^3(\Omega)$, the estimate \eqref{eq:ell_proj_est_2}
follows from the estimate derived for the PPIFE methods for elliptic problems in \cite{TLin_YLin_XZhang_PPIFE_Elliptic} again.
Similarly, we can obtain \eqref{eq:ell_proj_est_3}.
\end{proof}

\subsection{Error estimation for the semi-discrete method}

The {\it a priori} error estimates for semi-discrete PPIFE method \eqref{eq:DG-IFE_semi}-\eqref{eq:DG-IFE_semi_ic} for parabolic interface problem is given in the following theorem.
\begin{theorem}
Assume that the exact solution $u$ to the parabolic
 interface problem \eqref{eq:parab_eq}-\eqref{eq:parab_eq_jump_2} is in $H^1(0,T;\tilde H^3(\Omega))$ for $\epsilon=-1$ and in
 $H^2(0,T;\tilde H^3(\Omega))$ when $\epsilon=0, 1$, and $u_0\in\tilde H^3(\Omega)$.
 Let $u_h$ be the PPIFE solution defined by semi-discrete method \eqref{eq:DG-IFE_semi}-\eqref{eq:DG-IFE_semi_ic} with
 $\alpha=1$ and $u_h(\cdot,0) = \tilde u_h(\cdot)$ being the elliptic projection of $u_0$. Then there exists a constant $C$ such that
 \begin{equation}
\|u(\cdot,t)-u_h(\cdot,t)\|_h\leq Ch\Big(\|u_0\|_{\tilde H^3(\Omega)}+\|u_t\|_{L^2(0,T;\tilde H^3(\Omega))}\Big), \ \ \forall t\geq 0,
\label{eq:semi_est_1}
 \end{equation}
 for $\epsilon=-1$, and
  \begin{eqnarray}
 &&\|u(\cdot,t)-u_h(\cdot,t)\|_h\nonumber\\
 &&\leq Ch\Big(\|u_0\|_{\tilde H^3(\Omega)}+\|u_t(\cdot,0)\|_{\tilde H^3(\Omega)}+\|u_t\|_{L^2(0,T;\tilde H^3(\Omega))}+\|u_{tt}\|_{L^2(0,T;\tilde H^3(\Omega))}\Big), \ \ \forall t\geq 0, \label{eq:semi_est_2}
 \end{eqnarray}
 for $\epsilon=0$ or 1.
 \end{theorem}
\begin{proof}
Let $\tilde u_h$ be the elliptic projection of $u$ defined by \eqref{eq:ellip_proj} and we use it to split the
 error $u - u_h$ into two terms: $u - u_h = \eta - \xi$ with $\eta=u-\tilde u_h$ and $\xi=u_h-\tilde u_h$. For the first term, by \eqref{eq:ell_proj_est_1}, we have the following estimate:
 \begin{eqnarray}
 \|\eta(\cdot,t)\|_h&\leq& Ch\|u(\cdot,t)\|_{\tilde H^3(\Omega)}\leq Ch\Big(\|u_0\|_{\tilde H^3(\Omega)}+\int_0^t\|u_t\|_{\tilde H^3(\Omega)}d\tau\Big) \nonumber \\
 &\leq& Ch\Big(\|u_0\|_{\tilde H^3(\Omega)}+\|u_t\|_{L^2(0,T;\tilde H^3(\Omega))}\Big). \label{eq:semi_proof_1}
 \end{eqnarray}
Then, we proceed to bound $\|\xi\|_h$. From \eqref{eq:weak_form}, \eqref{eq:DG-IFE_semi} and \eqref{eq:ellip_proj}, we can see that $\xi$ satisfies the following equation:
\begin{equation}
 \left(\xi_t, v_h\right)+a_\epsilon(\xi,v_h)=\left(\eta_t, v_h\right),\ \ \forall
 v_h\in S_{h}(\Omega). \label{eq:semi_proof_2}
 \end{equation}
Choosing $v_h=\xi_t$ in \eqref{eq:semi_proof_2}, we have
\begin{equation}
 \|\xi_t\|^2+a_\epsilon(\xi, \xi_t)
 =\left(\eta_t, \xi_t\right). \label{eq:semi_proof_3}
 \end{equation}
\par
If $\epsilon=-1$, using the symmetry property of $a_\epsilon(\cdot,\cdot)$, Cauchy-Schwarz inequality and Young's inequality
in \eqref{eq:semi_proof_3},
we get
\begin{equation}
 \left\|\xi_t\right\|^2+\frac{1}{2}\frac{d}{dt}a_\epsilon(\xi,\xi)
 \leq\|\eta_t\|\left\|\xi_t\right\|\leq C\left\|\eta_t\right\|^2+\frac{1}{2}\left\|\xi_t\right\|^2. \label{eq:semi_proof_4}
 \end{equation}
For any $t\in (0,T]$, integrating both sides of \eqref{eq:semi_proof_4} from 0 to $t$, using the fact $\xi(\cdot,0)=0$ and \eqref{eq:ell_proj_est_2}, we obtain
\begin{equation}
\frac{1}{2}\int_0^t\left\|\xi_t\right\|^2+\frac{1}{2}a_\epsilon(\xi(\cdot,t),\xi(\cdot,t))
 \leq C\int_0^t\left\|\eta_t\right\|^2dt\leq Ch^2\int_0^T\left\|u_t\right\|_{\tilde H^3(\Omega)}^2. \label{eq:semi_proof_5}
\end{equation}
Using coercivity of $a_\epsilon(\cdot, \cdot)$ in \eqref{eq:semi_proof_5}, we have
\begin{equation}
\|\xi_t\|_{L^2(0,t;L^2(\Omega))}+\|\xi\|_h\leq Ch\|u_t\|_{L^2(0,T;\tilde H^3(\Omega))}. \label{eq:semi_proof_6}
\end{equation}
Finally, applying the triangle inequality, \eqref{eq:semi_proof_1} and \eqref{eq:semi_proof_6} to $u-u_h = \eta - \xi$ leads to \eqref{eq:semi_est_1}.\\

When $\epsilon=1$ or $0$,
$a_\epsilon(\cdot,\cdot)$ is not symmetric. However, we have
\begin{eqnarray}
a_\epsilon\left(\xi,
\xi_t\right)&=&\frac{1}{2}\frac{d}{dt}a_\epsilon(\xi,\xi)+\frac{1}{2}\left(a_\epsilon\left(\xi, \xi_t\right)-a_\epsilon\left(\xi_t,
\xi\right)\right) \nonumber \\
&\geq&\frac{1}{2}\frac{d}{dt}a_\epsilon(\xi,\xi)-C\|\xi_t\|_h\|\xi\|_h \nonumber \\
&\geq&\frac{1}{2}\frac{d}{dt}a_\epsilon(\xi,\xi)-\frac{C}{2}\left\|\xi_t\right\|_h^2-\frac{C}{2}\|\xi\|_h^2. \label{eq:semi_proof_7}
\end{eqnarray}
Substituting \eqref{eq:semi_proof_7} into \eqref{eq:semi_proof_3} and then integrating it from 0 to $t$, we can get
\begin{equation}
 \frac{1}{2}\int_0^t\|\xi_t\|^2d\tau+\frac{1}{2}\kappa\|\xi\|_h^2
 \leq
 C\int_0^t(\|\eta_t\|^2+\|\xi_t\|_h^2+\|\xi\|_h^2)d\tau. \label{eq:semi_proof_8}
\end{equation}
 Now we need the bound of
$\left\|\xi_t\right\|_h$. From \eqref{eq:semi_proof_2}, we can easily get
\begin{equation}
 \left(\xi_{tt}, v_h\right)+a_\epsilon(\xi_t, v_h)=\left(\eta_{tt}, v_h\right),\ \ \forall
 v_h\in S_{h}(\Omega), t\geq 0. \label{eq:semi_proof_9}
 \end{equation}
Choose $v_h=\xi_t$ in \eqref{eq:semi_proof_9} and use the coercivity of
$a_\epsilon(\cdot,\cdot)$ to get
\begin{eqnarray*}
 \frac{1}{2}\frac{d}{dt}\|\xi_{t}\|^2+\kappa \|\xi_t\|_h^2\leq
 \frac{1}{2}(\|\eta_{tt}\|^2+\|\xi_t\|^2).
\end{eqnarray*}
Integrating the above inequality from 0 to $t$ and using the Gronwall
inequality, we obtain
\begin{equation}
\int_0^t\|\xi_t\|_h^2d\tau\leq C\int_0^t\|\eta_{tt}\|^2d\tau+C\|\xi_t(\cdot,0)\|^2. \label{eq:semi_proof_10}
\end{equation}
Let $t=0$ and then choose $v_h=\xi_t(\cdot,0)$ in \eqref{eq:semi_proof_2} to get
\begin{equation}
\|\xi_t(\cdot,0)\|\leq \|\eta_t(\cdot,0)\|. \label{eq:semi_proof_11}
\end{equation}
Substituting \eqref{eq:semi_proof_10} and \eqref{eq:semi_proof_11} into \eqref{eq:semi_proof_8} and then using the Gronwall inequality again, we
obtain
\[
\int_0^t\|\xi_t\|^2d\tau+\|\xi\|_h^2
\leq C\int_0^t(\|\eta_t\|^2+\|\eta_{tt}\|^2)d\tau+C\|\eta_t(\cdot,0)\|^2.
\]
Applying \eqref{eq:ell_proj_est_2} and \eqref{eq:ell_proj_est_3} to the above yields
\begin{equation}
\left\|\xi_t\right\|_{L^2(0,t;L^2(\Omega))}+\|\xi\|_h\leq Ch\Big(\|u_t(\cdot,0)\|_{\tilde
 H^3(\Omega))}+\|u_{t}\|_{L^2(0,T;\tilde
 H^3(\Omega))}+\|u_{tt}\|_{L^2(0,T;\tilde
 H^3(\Omega))}\Big). \label{eq:semi_proof_12}
\end{equation}
Finally, applying the triangle inequality, \eqref{eq:semi_proof_1} and \eqref{eq:semi_proof_12} to $u-u_h = \eta - \xi$
yields \eqref{eq:semi_est_2}.
\end{proof}

\begin{remark}
By slightly modifying the proof for Theorem 3.1 we can show that estimates \eqref{eq:semi_est_1} and \eqref{eq:semi_est_2} still hold when
$\tilde u_h$ is chosen to be the IFE interpolation of $u_0$.
\end{remark}

\subsection{Error estimation for fully discrete methods}
In all the discussion from now on, we assume that $u_h^0=\tilde u_h$ is the elliptic projection of $u_0$ in the initial condition for the parabolic interface problem.
Also, for a function $\phi(t)$, we let $\phi^n = \phi(t^n), n \geq 0$.
\subsubsection{Backward Euler method}
The backward Euler method corresponds to the method described by \eqref{eq:DG-IFE_full_disc} with $\theta=1$.
From \eqref{eq:weak_form}, \eqref{eq:DG-IFE_full_disc} and \eqref{eq:ellip_proj}, we get
\begin{eqnarray}
&&\left(\partial_t \xi^n,
v_h\right)+a_\epsilon(\xi^{n},v_h)=\left(\partial_t \eta^n,
v_h\right)+(r^n, v_h),\ \ \forall v_h\in
S_{h}(\Omega), \label{eq:Euler_1}
\end{eqnarray}
where $r^n=-(u_t^n-\partial_t u^n )$.
We choose the test function $v_h=\partial_t\xi^n$ in \eqref{eq:Euler_1} and use the Cauchy-Schwarz inequality on the right hand side to obtain
 \begin{equation}
\|\partial_t\xi^n\|^2+
a_\epsilon(\xi^n,\partial_t\xi^{n})\leq\Big(\|\partial_t \eta^n\|+\|r^n\|\Big)\|\partial_t\xi^n\|\leq
\Big(\|\partial_t \eta^n\|^2+\|r^n\|^2\Big)+\frac{1}{2}\|\partial_t\xi^n\|^2. \label{eq:Euler_2}
\end{equation}
There are three cases depending on the parameter $\epsilon$. We start from the case in which $\epsilon=-1$. By the symmetry and the coercivity of the bilinear form $a_\epsilon(\cdot,\cdot)$, we have
\begin{eqnarray*}
a_\epsilon(\xi^n,\partial_t\xi^{n})&=&\frac{1}{\Delta t}a_\epsilon(\xi^n,\xi^n-\xi^{n-1})\\
&=&\frac{1}{2\Delta t}\Big(a_\epsilon(\xi^n,\xi^n)-a_\epsilon(\xi^{n-1},\xi^{n-1})\Big)+\frac{1}{2\Delta t}a_\epsilon(\xi^n-\xi^{n-1},\xi^n-\xi^{n-1})\\
&\geq&\frac{1}{2\Delta t}\Big(a_\epsilon(\xi^n,\xi^n)-a_\epsilon(\xi^{n-1},\xi^{n-1})\Big).
\end{eqnarray*}
Thus, we have
\begin{eqnarray}
&&\frac{1}{2}\|\partial_t\xi^n\|^2+
\frac{1}{2\Delta t}\Big(a_\epsilon(\xi^n,\xi^{n})-a_\epsilon(\xi^{n-1},\xi^{n-1})\Big)\leq\|\partial_t \eta^n\|^2+\|r^n\|^2. \label{eq:Euler_3}
\end{eqnarray}
Multiply \eqref{eq:Euler_3} by $2\Delta t$ and then sum over $n$ to get
\begin{eqnarray}
&&\Delta t \sum\limits_{n=1}^{k}\|\partial_t\xi^n\|^2+
a_\epsilon(\xi^k,\xi^{k})\leq2\Delta t\sum\limits_{n=1}^k\Big(\|\partial_t \eta^n\|^2+\|r^n\|^2\Big). \label{eq:Euler_4}
\end{eqnarray}
By H\"{o}lder's inequality and \eqref{eq:ell_proj_est_2}, we have
\begin{eqnarray}
\|\partial_t \eta^n\|^2&=&\int_{\Omega}\left(\frac{\eta^n-\eta^{n-1}}{\Delta t}\right)^2dX
= \int_{\Omega}\Big(\frac{1}{\Delta t}\int_{t^{n-1}}^{t^n}\eta_td\tau\Big)^2dX \nonumber \\
&\leq&\frac{1}{\Delta t}\int_{t^{n-1}}^{t^n}\|\eta_t\|^2d\tau\leq C\frac{h^2}{\Delta t}\int_{t^{n-1}}^{t^n}\|u_t\|_{\tilde H^3(\Omega)}^2d\tau. \label{eq:Euler_5}
\end{eqnarray}
Applying Taylor formula and H\"{o}lder's inequality, we have
\begin{equation}\label{eq:Euler_6}
  \|r^n\|^2=\int_{\Omega}|u_t^n-\partial_t u^n|^2dX=\int_{\Omega}\left|\frac{1}{\Delta t}\int_{t^{n-1}}^{t^n}(t-t^{n-1})u_{tt}dt\right|^2dX
\leq \frac{\Delta t}{3}\int_{t^{n-1}}^{t^n}\|u_{tt}\|^2d\tau.
\end{equation}
Substituting \eqref{eq:Euler_5} and \eqref{eq:Euler_6} into \eqref{eq:Euler_4} and then using the coercivity of $a_\epsilon(\cdot,\cdot)$, we obtain
 \begin{eqnarray}
&&\Delta t \sum\limits_{n=1}^{k}\|\partial_t\xi^n\|^2+
\|\xi^k\|_h^2\leq C\Big(h^2\|u_t\|_{L^2(0,T;\tilde H^3(\Omega))}^2+(\Delta t)^2\|u_{tt}\|_{L^2(0,T;L^2(\Omega))}^2\Big). \label{eq:Euler_7}
\end{eqnarray}
Finally, applying the triangle inequality, \eqref{eq:semi_proof_2} and \eqref{eq:Euler_7} to $u^k-u_h^k = \eta^k - \xi^k$ yields
\begin{equation}
\|u^k-u_h^k\|_h\leq C\Big(h\big(\|u_0\|_{\tilde H^3(\Omega)}+\|u_t\|_{L^2(0,T;\tilde H^3(\Omega))}\big)+\Delta t\|u_{tt}\|_{L^2(0,T;L^2(\Omega))}\Big)
\label{eq:Euler_8}
\end{equation}
for any integer $k\geq 0$.

\par
Now we turn to the cases where $\epsilon=0$ or $\epsilon = 1$ that make the bilinear form in the PPIFE methods nonsymmetric. We start from
\begin{eqnarray*}
a_\epsilon(\xi^n,\partial_t\xi^{n})&=&\frac{1}{\Delta t}a_\epsilon(\xi^n,\xi^n-\xi^{n-1})\\
&=&\frac{1}{2\Delta t}\Big(a_\epsilon(\xi^n,\xi^n)-a_\epsilon(\xi^{n-1},\xi^{n-1})\Big)+\frac{1}{2\Delta t}a_\epsilon(\xi^n,\xi^n-\xi^{n-1})-\frac{1}{2\Delta t}a_\epsilon(\xi^n-\xi^{n-1},\xi^{n-1})\\
&=&\frac{1}{2\Delta t}\Big(a_\epsilon(\xi^n,\xi^n)-a_\epsilon(\xi^{n-1},\xi^{n-1})\Big)+\frac{1}{2}\Big(a_\epsilon(\xi^n,\partial_t\xi^{n})-a_\epsilon(\partial_t\xi^{n},\xi^{n-1})\Big)\\
&\geq&\frac{1}{2\Delta t}\Big(a_\epsilon(\xi^n,\xi^n)-a_\epsilon(\xi^{n-1},\xi^{n-1})\Big)-C\Big(\|\partial_t\xi^n\|_h^2+\|\xi^{n-1}\|_h^2+\|\xi^{n}\|_h^2\Big).
\end{eqnarray*}
Substituting it into \eqref{eq:Euler_2} leads to
\begin{eqnarray}
&&\frac{1}{2}\|\partial_t\xi^n\|^2+
\frac{1}{2\Delta t}\Big(a_\epsilon(\xi^n,\xi^{n})-a_\epsilon(\xi^{n-1},\xi^{n-1})\Big) \nonumber \\
&&\qquad\leq\|\partial_t \eta^n\|^2
+\|r^n\|^2+C\Big(\|\partial_t\xi^n\|_h^2+\|\xi^{n-1}\|_h^2+\|\xi^{n}\|_h^2\Big). \label{eq:Euler_9}
\end{eqnarray}
Multiply \eqref{eq:Euler_9} by $2\Delta t$ and sum over $n$ to obtain
\begin{eqnarray}
&&\sum\limits_{n=1}^k\Delta t\|\partial_t\xi^n\|^2+
\kappa\|\xi^k\|_h^2\leq\sum\limits_{n=1}^k\Delta t(\|\partial_t \eta^n\|^2
+\|r^n\|^2)+C\sum\limits_{n=1}^k\Delta t\|\partial_t\xi^n\|_h^2 + C\sum\limits_{n=1}^k\Delta t\|\xi^{n}\|_h^2. \label{eq:Euler_10}
\end{eqnarray}
In order to bound $\sum\limits_{n=1}^k\Delta t\|\partial_t\xi^n\|_h^2$, we first derive from \eqref{eq:Euler_1} that
\begin{eqnarray}
&&\frac{1}{\Delta t}\left(\partial_t \xi^n-\partial_t\xi^{n-1},
v_h\right)+a_\epsilon(\partial_t\xi^{n},v_h)=\left(\partial_{tt} \eta^n,
v_h\right)+(\partial_t r^n, v_h),\ \ \forall v_h\in
S_{h}(\Omega). \label{eq:Euler_11}
\end{eqnarray}
Let $v_h=\partial_t \xi^n$ in \eqref{eq:Euler_11} to get
\begin{eqnarray*}
&&\frac{1}{2\Delta t}\left(\|\partial_t \xi^n\|^2-\|\partial_t\xi^{n-1}\|^2
\right)+\kappa\|\partial_t\xi^{n}\|_h^2\leq (\|\partial_{tt} \eta^n\|
+\|\partial_t r^n\|)\|\partial_t \xi^n\|.
\end{eqnarray*}
Then we can easily obtain
\begin{eqnarray}
&&\|\partial_t \xi^k\|^2
+\sum\limits_{n=2}^k\Delta t\|\partial_t\xi^{n}\|_h^2\leq C\sum\limits_{n=2}^k\Delta t(\|\partial_{tt} \eta^n\|^2
+\|\partial_t r^n\|^2)+C\|\partial_t \xi^1\|^2. \label{eq:Euler_12}
\end{eqnarray}
Let $n=1$ and $v_h=\partial_t \xi^1=\xi^1/\Delta t$ in \eqref{eq:Euler_1}, then we have
\begin{eqnarray*}
&&\|\partial_t \xi^1\|^2+
\frac{1}{\Delta t}a_\epsilon(\xi^1,\xi^{1})\leq(\|\partial_t \eta^1\|+\|r^1\|)\|\partial_t \xi^1\|.
\end{eqnarray*}
Thus
\begin{eqnarray*}
&&\|\partial_t \xi^1\|^2+\frac{1}{\Delta t}\|\xi^1\|_h^2\leq C(\|\partial_t \eta^1\|^2+\|r^1\|^2).
\end{eqnarray*}
Applying this to \eqref{eq:Euler_12} yields
\begin{eqnarray}
&&\sum\limits_{n=1}^k\Delta t\|\partial_t\xi^{n}\|_h^2\leq C\sum\limits_{n=2}^k\Delta t(\|\partial_{tt} \eta^n\|^2
+\|\partial_t r^n\|^2)+C(\|\partial_t \eta^1\|^2+\|r^1\|^2). \label{eq:Euler_13}
\end{eqnarray}
Inserting \eqref{eq:Euler_13} into \eqref{eq:Euler_10}, then applying the Gronwall inequality, we obtain
\begin{eqnarray}
&&\|\xi^k\|_h^2 \leq C\sum\limits_{n=1}^k\Delta t(\|\partial_t \eta^n\|^2
+\|r^n\|^2)+C\sum\limits_{n=2}^k\Delta t(\|\partial_{tt} \eta^n\|^2
+\|\partial_t r^n\|^2)+C(\|\partial_t \eta^1\|^2+\|r^1\|^2). \label{eq:Euler_14}
\end{eqnarray}
We now estimate the last four terms in \eqref{eq:Euler_14}. It is easily to see that
\begin{eqnarray*}\nonumber
\|\partial_{tt} \eta^n\|^2&=&\int_{\Omega}\left(\frac{\eta^n-2\eta^{n-1}+\eta^{n-2}}{(\Delta t)^2}\right)^2dX\\
&=&\int_{\Omega}\left(\frac{1}{(\Delta t)^2}\int_{t^{n-1}}^{t^n}\eta_{tt}(t^n-t)dt-\frac{1}{(\Delta t)^2}\int_{t^{n-2}}^{t^{n-1}}\eta_{tt}(t^{n-1}-t)dt\right)^2dX\\
&\leq&\frac{1}{3\Delta t}\int_{t^{n-2}}^{t^n}\|\eta_{tt}\|^2dt.
\end{eqnarray*}
This inequality and \eqref{eq:ell_proj_est_3} lead to
\begin{eqnarray}
\sum\limits_{n=2}^k\Delta t\|\partial_{tt} \eta^n\|^2\leq Ch^2\|u_{tt}\|^2_{L^2(0,T;\tilde H^3(\Omega))}. \label{eq:Euler_15}
\end{eqnarray}
Also, we have
\begin{eqnarray*}\nonumber
\partial_{t}r^n&=&\frac{u_t^n-u_t^{n-1}}{\Delta t}-\frac{u^n-2u^{n-1}+u^{n-2}}{(\Delta t)^2}\\
&=&\int_{t^{n-1}}^{t^n}u_{ttt}dt-\frac{1}{(\Delta
t)^2}\int_{t^{n-1}}^{t^n}u_{ttt}(t^{n-1}-t)^2dt+\frac{1}{(\Delta
t)^2}\int_{t^{n-2}}^{t^{n-1}}u_{ttt}(t-t^{n-2})^2dt.
\end{eqnarray*}
Then, the application of H\"{o}lder's inequality leads to
\begin{eqnarray}
\sum\limits_{n=2}^k\Delta t\|\partial_t r^n\|^2&\leq& (\Delta
t)^2\sum\limits_{n=2}^k\left(\int_{t^{n-1}}^{t^n}\|u_{ttt}\|^2dt + \frac{1}{5}\int_{t^{n-1}}^{t^n}\|u_{ttt}\|^2dt
+\frac{1}{5}\int_{t^{n-2}}^{t^{n-1}}\|u_{ttt}\|^2dt\right) \nonumber\\
&\leq& C(\Delta t)^2\|u_{ttt}\|^2_{L^2(0,T;L^2(\Omega))}. \label{eq:Euler_16}
\end{eqnarray}
As for the last two terms on the right hand side of \eqref{eq:Euler_14}, we have
\begin{eqnarray}
\|\partial_t \eta^1\|^2\leq \frac{1}{\Delta t}\int_0^{\Delta
t}\|\eta_t\|^2dt\leq h^2\left(\frac{1}{\Delta t}\int_0^{\Delta
t}\|u_t\|_{\tilde H^3(\Omega)}^2dt\right) \label{eq:Euler_17}
\end{eqnarray}
and, by \eqref{eq:Euler_6},
 \begin{eqnarray}
 \|r^1\|^2=\int_{\Omega}|u_t^1-\partial_t u^1|^2dX\leq \frac{(\Delta t)^2}{3}\left(\frac{1}{\Delta t}\int_{0}^{\Delta t}\|u_{tt}\|^2dt\right).
 \label{eq:Euler_18}
 \end{eqnarray}
 Now, substituting \eqref{eq:Euler_5}, \eqref{eq:Euler_6} and \eqref{eq:Euler_15}-\eqref{eq:Euler_18} into \eqref{eq:Euler_14}, we
 obtain
\begin{eqnarray*}\nonumber
\|\xi^k\|_h^2&\leq& C\left(\|u_t\|_{L^2(0,T;\tilde
H^3(\Omega))}^2+\|u_{tt}\|^2_{L^2(0,T;\tilde
H^3(\Omega))}+\frac{1}{\Delta t}\int_0^{\Delta t}\|u_t\|_{\tilde
H^3(\Omega)}^2dt\right)h^2\\
&&+C\left(\|u_{tt}\|_{L^2(0,T;L^2(\Omega))}^2+\|u_{ttt}\|^2_{L^2(0,T;L^2(\Omega))}+\frac{1}{\Delta
t}\int_{0}^{\Delta t}\|u_{tt}\|^2dt\right)(\Delta t)^2.
\end{eqnarray*}
Again, applying the estimate for $\xi^k$, the triangle inequality and \eqref{eq:semi_proof_1} to
$u^k - u_h^k = \eta^k - \xi^k$, we obtain
\begin{eqnarray*}\nonumber
&&\|u^k-u_h^k\|_h\\
&\leq& C\left(\|u_0\|_{\tilde
H^3(\Omega)}+\|u_t\|_{L^2(0,T;\tilde
H^3(\Omega))}+\|u_{tt}\|_{L^2(0,T;\tilde
H^3(\Omega))}+\left(\frac{1}{\Delta t}\int_0^{\Delta
t}\|u_t\|_{\tilde
H^3(\Omega)}^2dt\right)^{1/2}\right)h\\
&&+C\left(\|u_{tt}\|_{L^2(0,T;L^2(\Omega))}+\|u_{ttt}\|_{L^2(0,T;L^2(\Omega))}+\left(\frac{1}{\Delta
t}\int_{0}^{\Delta t}\|u_{tt}\|^2dt\right)^{1/2}\right)\Delta t.
\end{eqnarray*}
Now let us summarize the analysis above for the backward Euler PPIFE method in the following theorem.

\begin{theorem}
Assume that the exact solution $u$ to the parabolic interface problem \eqref{eq:parab_eq}-\eqref{eq:parab_eq_jump_2} is in $H^2(0,T;\tilde H^3(\Omega))\cap H^3(0,T;L^2(\Omega))$ and
 $u_0\in \tilde H^3(\Omega)$.
 Let the sequence $\Big\{u_h^n\Big\}_{n=0}^{N_t}$ be the solution to
the backward Euler PPIFE method \eqref{eq:DG-IFE_full_disc}-\eqref{eq:DG-IFE_full_disc_ic}. Then, we have the following estimates:
\begin{description}
  \item[(1)] If $\epsilon=-1$, then there exists a positive constant $C$ independent of $h$ and $\Delta t$ such that
   \begin{equation}
\max_{0\leq n\leq N_t}\|u^n-u_h^n\|_h\leq C\Big(h\big(\|u_0\|_{\tilde H^3(\Omega)}+\|u_t\|_{L^2(0,T;\tilde H^3(\Omega))}\big)+\Delta t\|u_{tt}\|_{L^2(0,T;L^2(\Omega))}\Big). \label{eq:Euler_19}
\end{equation}
  \item[(2)] If $\epsilon=0$ or 1, then there exists a positive constant $C$ independent of $h$ and $\Delta t$ such that
  \begin{eqnarray}\nonumber
&&\max_{0\leq n\leq N_t}\|u^n-u_h^n\|_h\\ \nonumber
&\leq& C\left(\|u_0\|_{\tilde
H^3(\Omega)}+\|u_t\|_{L^2(0,T;\tilde
H^3(\Omega))}+\|u_{tt}\|_{L^2(0,T;\tilde
H^3(\Omega))}+\left(\frac{1}{\Delta t}\int_0^{\Delta
t}\|u_t\|_{\tilde
H^3(\Omega)}^2dt\right)^{1/2}\right)h\\
&&+C\left(\|u_{tt}\|_{L^2(0,T;L^2(\Omega))}+\|u_{ttt}\|_{L^2(0,T;L^2(\Omega))}+\left(\frac{1}{\Delta
t}\int_{0}^{\Delta t}\|u_{tt}\|^2dt\right)^{1/2}\right)\Delta t. \label{eq:Euler_20}
\end{eqnarray}
\end{description}
\end{theorem}
\subsubsection{Crank-Nicolson method}
Now we conduct the  error analysis for the  Crank-Nicolson method which corresponds to $\theta=1/2$ in \eqref{eq:DG-IFE_full_disc}.
From \eqref{eq:weak_form}, \eqref{eq:DG-IFE_full_disc} and \eqref{eq:ellip_proj}, we have
\begin{eqnarray}
&&\left(\partial_t \xi^n,
v_h\right)+\frac{1}{2}a_\epsilon(\xi^{n}+\xi^{n-1},v_h)=(\partial_t \eta^n,
v_h)+(r_1^n, v_h)+(r_2^n, v_h),\ \ \forall v_h\in
S_{h}(\Omega), \label{eq:CN_21}
\end{eqnarray}
where
\begin{eqnarray*}
&&r_1^n=u_t^{n-1/2}-\frac{1}{2}(u_t^n+u_t^{n-1}),~~r_2^n=-(u_t^{n-1/2}-\partial_t u^n).
\end{eqnarray*}
Taking $v_h=\partial_t \xi^n=(\xi^{n}-\xi^{n-1})/\Delta t$ in \eqref{eq:CN_21} and applying the Cauchy-Schwarz inequality, we get
\begin{eqnarray}
\|\partial_t \xi^n\|^2+
\frac{1}{2\Delta t}a_\epsilon(\xi^{n}+\xi^{n-1},\xi^{n}-\xi^{n-1})&\leq&\Big(\|\partial_t \eta^n\|
+\|r_1^n\|+\|r_2^n\|\Big)\|\partial_t \xi^n\| \nonumber \\
&\leq&C\Big(\|\partial_t \eta^n\|^2
+\|r_1^n\|^2+\|r_2^n\|^2\Big)+\frac{1}{2}\|\partial_t \xi^n\|^2. \label{eq:CN_22}
\end{eqnarray}
If $\epsilon=-1$, due to the
 symmetry of $a_\epsilon(\cdot,\cdot)$, we can rewrite \eqref{eq:CN_22} as
\begin{eqnarray}
\|\partial_t \xi^n\|^2+
\frac{1}{2\Delta t}\Big(a_\epsilon(\xi^{n},\xi^{n})-a_\epsilon(\xi^{n-1},\xi^{n-1})\Big) &\leq& C\Big(\|\partial_t \eta^n\|^2 +\|r_1^n\|^2+\|r_2^n\|^2\Big). \label{eq:CN_23}
\end{eqnarray}
Multiplying \eqref{eq:CN_23} by $2\Delta t$ and summing over $n$, we have
\begin{eqnarray}
\kappa\|\xi^k\|_h^2\leq a_\epsilon(\xi^{k},\xi^{k})&\leq&C\sum\limits_{n=1}^k\Delta t\Big(\|\partial_t \eta^n\|^2
+\|r_1^n\|^2+\|r_2^n\|^2\Big). \label{eq:CN_24}
\end{eqnarray}
We note that \eqref{eq:Euler_5} is still a valid estimation for $\|\partial_t \eta^n\|^2$; hence, we proceed to estimate $\|r_1^n\|^2$ and $\|r_2^n\|^2$.
From the Taylor formula and H\"{o}lder's inequality, we obtain
\begin{eqnarray}\nonumber
\|r_1^n\|^2&=&\int_{\Omega}\Big(u_t^{n-1/2}-\frac{1}{2}(u_t^n+u_t^{n-1})\Big)^2dX\\ \nonumber
&=&\int_{\Omega}\frac{1}{4}\left(\int_{t^{n-1}}^{t^{n-1/2}}u_{ttt}(t-t^{n-1})dt+\int_{t^{n-1/2}}^{t^{n}}u_{ttt}(t^n-t)dt\right)^2dX\\
&\leq& C(\Delta t)^3\int_{t^{n-1}}^{t^{n}}\|u_{ttt}\|^2dt, \label{eq:CN_25}
\end{eqnarray}
and
\begin{eqnarray}\nonumber
\|r_2^n\|^2&=&\int_{\Omega}\Big(u_t^{n-1/2}-\partial_t u^n)\Big)^2dX\\ \nonumber
&=&\int_{\Omega}\frac{1}{(\Delta t)^2}\left(\int_{t^{n-1}}^{t^{n-1/2}}u_{ttt}(t-t^{n-1})^2dt+\int_{t^{n-1/2}}^{t^{n}}u_{ttt}(t^n-t)^2dt\right)^2dX\\
&\leq& C(\Delta t)^3\int_{t^{n-1}}^{t^{n}}\|u_{ttt}\|^2dt. \label{eq:CN_26}
\end{eqnarray}
Using \eqref{eq:Euler_5}, \eqref{eq:CN_25} and \eqref{eq:CN_26} in \eqref{eq:CN_24} yields
\[
\|\xi^k\|_h^2\leq C\Big(h^2\|u_t\|^2_{L^2(0,T;\tilde H^3(\Omega))}+(\Delta t)^4\|u_{ttt}\|^2_{L^2(0,T;L^2(\Omega))}\Big).
\]
Finally, we obtain an estimate for $u^k - u_h^k$ by applying the above estimate for $\xi^k$, the triangle inequality and \eqref{eq:ell_proj_est_1} to
the splitting $u^k - u_h^k = \eta^k - \xi^k$, and we summarize the result in the following theorem.\\

\noindent
\begin{theorem}
Assume that the exact solution $u$ to the parabolic interface problem \eqref{eq:parab_eq}-\eqref{eq:parab_eq_jump_2} is in $H^1(0,T;\tilde H^3(\Omega))\cap H^3(0,T;L^2(\Omega))$ and $u_0\in \tilde H^3(\Omega)$. Assume the sequence $\Big\{u_h^n\Big\}_{n=0}^{N_t}$ is the solution to
the PPIFE Crank-Nicolson method \eqref{eq:DG-IFE_full_disc}-\eqref{eq:DG-IFE_full_disc_ic} with $\epsilon=-1$. Then, there exists a positive constant $C$ independent of $h$ and $\Delta t$ such that
\begin{equation}\label{eq: CN_error}
\max_{0\leq n\leq N_t}\|u^n-u_h^n\|_h\leq C\Big(h(\|u_0\|_{\tilde
H^3(\Omega)}+\|u_t\|_{L^2(0,T;\tilde
H^3(\Omega))})+(\Delta t)^2\|u_{ttt}\|_{L^2(0,T;L^2(\Omega))}\Big).
\end{equation}
\end{theorem}

\begin{remark}
The choice of $\epsilon = -1$ for the PPIFE Crank-Nicolson method is very natural because the method inherits the symmetry from the interface problem and
its algebraic system is easier to solve. On the other hand, even though the non-symmetric PPIFE Crank-Nicolson methods based on the other two choices of $\epsilon = 0$ and
$\epsilon = 1$ also seem to work well as demonstrated by the numerical results in the next section, the asymmetry in their bilinear forms hinders the
estimation of several key terms in the error analysis so that the related convergence still remains elusive.
\end{remark}

\begin{remark} We can replace the bilinear form $a_\epsilon(\cdot, \cdot)$ with the one used in the standard interior penalty DG finite element
methods to obtain corresponding DGIFE methods for the parabolic interface problems. Furthermore, the error estimation for PPIFE methods can also be readily extended to the corresponding DGIFE methods. However, as usual, these DGIFE methods have much more unknowns than the PPIFE counterparts; hence they are less favorable unless features in DG formulation are desired.
\end{remark}

\section{Numerical Examples}
\setcounter{equation}{0}

\noindent
In this section, we present some numerical results to demonstrate features of PPIFE methods for parabolic interface problems. \vspace{2mm}

Let the solution domain be $\Omega=(0,1)\times(0,1)$ and the time interval be $[0,1]$. The interface curve $\Gamma$ is chosen to be an ellipse centered at the point $(x_0,y_0)$ with semi-radius $a$ and $b$, whose parametric form can be written as
\begin{equation}\label{eq: ellipse parametric form}
  \left\{
    \begin{array}{ll}
      x = x_0 + a\cos(\theta), \\
      y = y_0 + b\sin(\theta).
    \end{array}
  \right.
\end{equation}
In our numerical experiments, we choose $x_0 = y_0 = 0$, $a = \pi/4$, $b=\pi/6$, and $\theta\in [0,\pi/2]$.
The interface $\Gamma$ separates $\Omega$ into two sub-domains $\Omega^- = \{(x,y):r(x,y)<1\}$ and $\Omega^+= \{(x,y):r(x,y)>1\}$
where
\begin{equation*}
  r(x,y) = \sqrt{\frac{(x-x_0)^2}{a^2} + \frac{(y-y_0)^2}{b^2}}.
\end{equation*}
The exact solution for the parabolic interface problem is chosen to be
\begin{equation}\label{eq: true solution}
    u(t,x,y) =
    \left\{
      \begin{array}{ll}
        \frac{1}{\beta^-}r^p e^t, \quad& \text{if~} (x,y)\in \Omega^- , \\
        \left(\frac{1}{\beta^+}r^p - \frac{1}{\beta^+} + \frac{1}{\beta^-}\right) e^t, \quad& \text{if~} (x,y)\in \Omega^+,
      \end{array}
    \right.
\end{equation}
where $p = 5$ and the diffusion coefficients $\beta^\pm$ vary in different numerical experiments.

We use a family of  Cartesian meshes $\{\mathcal{T}_h, h>0\}$, and each mesh is formed by partitioning $\Omega$ into $N_s\times N_s$ congruent squares of size $h = 1/N_s$ for a set of values of integer $N_s$. For fully discretized methods, we divide the time interval $[0,1]$ into $N_t$ subintervals uniformly with $t^n = n\Delta t$, $n=0,1,\cdots, N_t$, and $\Delta t = 1/N_t$. Also, we have observed that the condition numbers of the matrices associated with the bilinear forms in these IFE methods is proportional to
$h^{-2}$, similar to that of the standard finite element method; therefore, usual solvers can be applied to efficiently solve the sparse linear system in these IFE methods.

\par
First, we consider the case in which the diffusion coefficient $(\beta^-,\beta^+) = (1,10)$ representing a moderate discontinuity across the interface. Both nonsymmetric ($\epsilon = 1$) and symmetric ($\epsilon = -1$) PPIFE methods are  employed to solve the parabolic interface problem. For penalty parameters, we choose $\sigma_B^0 = 1$ for the nonsymmetric method and $\sigma_B^0 = 100$ for the symmetric method, while $\alpha = 1$ for both methods. Both backward Euler and Crank-Nicolson methods are employed and the time step is chosen as $\Delta t = 2h$. Errors of nonsymmetric and symmetric PPIFE backward Euler methods in $L^\infty$, $L^2$ and semi-$H^1$ norms are listed in Table \ref{table: NPPIFE error 1 10 tau h BE exp} and Table \ref{table: SPPIFE error 1 10 tau h BE exp}, respectively. Errors of nonsymmetric and symmetric PPIFE Crank-Nicolon methods are listed in Table \ref{table: NPPIFE error 1 10 tau h CN exp} and Table \ref{table: SPPIFE error 1 10 tau h CN exp}, respectively. All errors are computed at the final time level, \emph{i.e.} $t=1$.

\begin{table}[htb]
\begin{center}
\begin{tabular}{|c|cc|cc|cc|}
\hline
$h$
& $\|\cdot\|_{L^\infty}$ & rate
& $\|\cdot\|_{L^2}$ & rate
& $|\cdot|_{H^1}$ & rate \\
\hline
$1/10$  &$2.7866E{-2}$&        &$8.2619E{-2}$&        &$2.1079E{-0}$&         \\
$1/20$  &$7.9371E{-3}$& 1.8118 &$2.0935E{-2}$& 1.9805 &$1.0659E{-0}$& 0.9838  \\
$1/40$  &$2.6530E{-3}$& 1.5810 &$5.3984E{-3}$& 1.9553 &$5.3875E{-1}$& 0.9844  \\
$1/80$  &$8.9636E{-4}$& 1.5655 &$1.4473E{-3}$& 1.8991 &$2.7065E{-1}$& 0.9932 \\
$1/160$ &$3.3405E{-4}$& 1.4240 &$4.1586E{-4}$& 1.7992 &$1.3567E{-1}$& 0.9963 \\
$1/320$ &$1.3871E{-4}$& 1.2680 &$1.3204E{-4}$& 1.6551 &$6.7927E{-2}$& 0.9980  \\
$1/640$ &$6.2344E{-5}$& 1.1538 &$4.7909E{-5}$& 1.4626 &$3.3986E{-2}$& 0.9990 \\
$1/1280$&$2.9411E{-5}$& 1.0839 &$1.9763E{-5}$& 1.2775 &$1.6998E{-2}$& 0.9996  \\
\hline
\end{tabular}
\end{center}
\caption{Errors of nonsymmetric PPIFE backward Euler solutions with $\beta^- = 1$, $\beta^+ = 10$ at time $t=1$}
\label{table: NPPIFE error 1 10 tau h BE exp}
\end{table}

\begin{table}[htb]
\begin{center}
\begin{tabular}{|c|cc|cc|cc|}
\hline
$h$
& $\|\cdot\|_{L^\infty}$ & rate
& $\|\cdot\|_{L^2}$ & rate
& $|\cdot|_{H^1}$ & rate \\
\hline
$1/10$  &$6.6821E{-2}$&        &$8.1952E{-2}$&        &$2.1051E{-0}$&         \\
$1/20$  &$1.5332E{-2}$& 2.1237 &$2.1070E{-2}$& 1.9596 &$1.0654E{-0}$& 0.9826  \\
$1/40$  &$5.1586E{-3}$& 1.5715 &$5.4326E{-3}$& 1.9554 &$5.3876E{-1}$& 0.9836  \\
$1/80$  &$1.5387E{-3}$& 1.7453 &$1.4582E{-3}$& 1.8974 &$2.7067E{-1}$& 0.9931 \\
$1/160$ &$4.9034E{-4}$& 1.6498 &$4.1727E{-4}$& 1.8052 &$1.3567E{-1}$& 0.9964 \\
$1/320$ &$1.7632E{-4}$& 1.4755 &$1.3212E{-4}$& 1.6591 &$6.7927E{-2}$& 0.9980  \\
$1/640$ &$7.1949E{-5}$& 1.2932 &$4.7927E{-5}$& 1.4630 &$3.3986E{-2}$& 0.9990 \\
$1/1280$&$3.1775E{-5}$& 1.1791 &$1.9763E{-5}$& 1.2780 &$1.6998E{-2}$& 0.9996  \\
\hline
\end{tabular}
\end{center}
\caption{Errors of symmetric PPIFE backward Euler solutions with $\beta^- = 1$, $\beta^+ = 10$ at time $t=1$}
\label{table: SPPIFE error 1 10 tau h BE exp}
\end{table}

\begin{table}[htb]
\begin{center}
\begin{tabular}{|c|cc|cc|cc|}
\hline
$h$
& $\|\cdot\|_{L^\infty}$ & rate
& $\|\cdot\|_{L^2}$ & rate
& $|\cdot|_{H^1}$ & rate \\
\hline
$1/10$  &$5.1829E{-2}$&        &$9.3610E{-2}$&        &$2.1106E{-0}$&         \\
$1/20$  &$1.0369E{-2}$& 2.3215 &$2.2475E{-2}$& 2.0583 &$1.0658E{-0}$& 0.9857  \\
$1/40$  &$2.8024E{-3}$& 1.8875 &$5.6292E{-3}$& 1.9973 &$5.3870E{-1}$& 0.9843  \\
$1/80$  &$7.1649E{-4}$& 1.9676 &$1.4091E{-3}$& 1.9982 &$2.7063E{-1}$& 0.9931 \\
$1/160$ &$1.7881E{-4}$& 2.0026 &$3.5445E{-4}$& 1.9911 &$1.3566E{-1}$& 0.9963 \\
$1/320$ &$4.5518E{-5}$& 1.9739 &$8.8742E{-5}$& 1.9979 &$6.7926E{-2}$& 0.9980  \\
$1/640$ &$1.1447E{-5}$& 1.9914 &$2.2156E{-5}$& 2.0019 &$3.3986E{-2}$& 0.9990 \\
$1/1280$&$2.8833E{-6}$& 1.9892 &$5.5375E{-7}$& 2.0004 &$1.6998E{-2}$& 0.9996  \\
\hline
\end{tabular}
\end{center}
\caption{Errors of nonsymmetric PPIFE Crank-Nicolson solutions with $\beta^- = 1$, $\beta^+ = 10$ at time $t=1$}
\label{table: NPPIFE error 1 10 tau h CN exp}
\end{table}

\begin{table}[htb]
\begin{center}
\begin{tabular}{|c|cc|cc|cc|}
\hline
$h$
& $\|\cdot\|_{L^\infty}$ & rate
& $\|\cdot\|_{L^2}$ & rate
& $|\cdot|_{H^1}$ & rate \\
\hline
$1/10$  &$1.0310E{-1}$&        &$9.2384E{-2}$&        &$2.1112E{-0}$&         \\
$1/20$  &$1.4252E{-2}$& 2.8447 &$2.2543E{-2}$& 2.0350 &$1.0650E{-0}$& 0.9872  \\
$1/40$  &$4.2963E{-3}$& 1.7401 &$5.6546E{-3}$& 1.9952 &$5.3862E{-1}$& 0.9836  \\
$1/80$  &$1.0893E{-4}$& 1.9796 &$1.4190E{-3}$& 1.9946 &$2.7062E{-1}$& 0.9930 \\
$1/160$ &$2.8178E{-4}$& 1.9508 &$3.5605E{-4}$& 1.9947 &$1.3566E{-1}$& 0.9963 \\
$1/320$ &$6.6903E{-5}$& 2.0744 &$8.9021E{-5}$& 1.9999 &$6.7924E{-2}$& 0.9980  \\
$1/640$ &$1.6838E{-5}$& 1.9903 &$2.2251E{-5}$& 2.0003 &$3.3985E{-2}$& 0.9990 \\
$1/1280$&$4.1878E{-6}$& 2.0075 &$5.5633E{-6}$& 1.9999 &$1.6998E{-2}$& 0.9996  \\
\hline
\end{tabular}
\end{center}
\caption{Errors of symmetric PPIFE Crank-Nicolson solutions with $\beta^- = 1$, $\beta^+ = 10$ at time $t=1$}
\label{table: SPPIFE error 1 10 tau h CN exp}
\end{table}

In Table \ref{table: NPPIFE error 1 10 tau h BE exp} and Table \ref{table: SPPIFE error 1 10 tau h BE exp}, we note that errors in semi-$H^1$ norms for both nonsymmetric and symmetric PPIFE backward Euler methods demonstrate an optimal convergence rate $O(h) + O(\Delta t)$, which confirms our error estimates \eqref{eq:Euler_19} and \eqref{eq:Euler_20}. Also note that the order of convergence in $L^2$ norm approaches $1$ as we perform uniform mesh refinement. This is consistent with our expectation of the order of convergence $O(h^2) + O(\Delta t)$ in $L^2$ norm although such an error bound has not been established yet. Errors gauged in $L^\infty$ norm indicate a first order convergence for backward Euler method.

In Table \ref{table: NPPIFE error 1 10 tau h CN exp} and Table \ref{table: SPPIFE error 1 10 tau h CN exp}, the convergence rate in semi-$H^1$ norm confirms our error estimate \eqref{eq: CN_error} for Crank-Nicolson method. Moreover, errors in $L^2$ norm is of second order convergence which agrees with our anticipated convergence rate $O(h^2) + O(\Delta t^2)$. Errors in $L^\infty$ norm also seem to maintain an optimal second order convergence.

\begin{table}[ht]
\begin{center}
\begin{tabular}{|c|cc|cc|cc|}
\hline
$h$
& $\|\cdot\|_{L^\infty}$ & rate
& $\|\cdot\|_{L^2}$ & rate
& $|\cdot|_{H^1}$ & rate \\
\hline
$1/10$  &$1.4637E{-1}$&        &$4.7718E{-2}$&        &$1.1268E{-0}$&         \\
$1/20$  &$6.4974E{-2}$& 1.1717 &$1.6100E{-2}$& 1.5675 &$5.9288E{-1}$& 0.9265  \\
$1/40$  &$2.2137E{-2}$& 1.5534 &$4.3284E{-3}$& 1.8951 &$3.0548E{-1}$& 0.9567  \\
$1/80$  &$7.2728E{-3}$& 1.6059 &$8.4067E{-4}$& 2.3642 &$1.5187E{-1}$& 1.0083 \\
$1/160$ &$2.3746E{-3}$& 1.6148 &$2.0844E{-4}$& 2.0119 &$7.5576E{-2}$& 1.0068 \\
$1/320$ &$1.0006E{-3}$& 1.2468 &$5.2912E{-5}$& 1.9779 &$3.7807E{-2}$& 0.9993  \\
$1/640$ &$1.7030E{-4}$& 2.5547 &$1.4993E{-5}$& 1.8193 &$1.8900E{-2}$& 1.0002 \\
$1/1280$&$6.3452E{-5}$& 1.4244 &$4.9410E{-6}$& 1.6014 &$9.4461E{-3}$& 1.0006  \\
\hline
\end{tabular}
\end{center}
\caption{Errors of nonsymmetric PPIFE backward Euler solutions with $\beta^- = 1$, $\beta^+ = 10000$ at time $t=1$}
\label{table: NPPIFE error 1 10000 tau h BE exp}
\end{table}


\begin{table}[ht]
\begin{center}
\begin{tabular}{|c|cc|cc|cc|}
\hline
$h$
& $\|\cdot\|_{L^\infty}$ & rate
& $\|\cdot\|_{L^2}$ & rate
& $|\cdot|_{H^1}$ & rate \\
\hline
$1/10$  &$1.9919E{-1}$&        &$5.2179E{-2}$&        &$1.1724E{-0}$&         \\
$1/20$  &$4.8082E{-2}$& 2.0506 &$1.5609E{-2}$& 1.7411 &$5.7800E{-1}$& 1.0204  \\
$1/40$  &$1.4716E{-2}$& 1.7081 &$4.2141E{-3}$& 1.8890 &$2.9879E{-1}$& 0.9519  \\
$1/80$  &$5.0467E{-3}$& 1.5439 &$8.1261E{-4}$& 2.3746 &$1.4997E{-1}$& 0.9945 \\
$1/160$ &$1.6228E{-3}$& 1.6368 &$1.9588E{-4}$& 2.0526 &$7.5188E{-2}$& 0.9961 \\
$1/320$ &$6.8515E{-3}$& 1.2440 &$4.5716E{-5}$& 2.0992 &$3.7698E{-2}$& 0.9960  \\
$1/640$ &$1.2256E{-4}$& 2.4830 &$1.0915E{-5}$& 2.0664 &$1.8871E{-2}$& 0.9983 \\
$1/1280$&$4.4326E{-5}$& 1.4672 &$2.6715E{-6}$& 2.0306 &$9.4387E{-3}$& 0.9995  \\
\hline
\end{tabular}
\end{center}
\caption{Errors of nonsymmetric PPIFE Crank-Nicolson solutions with $\beta^- = 1$, $\beta^+ = 10000$ at time $t=1$}
\label{table: NPPIFE error 1 10000 tau h CN exp}
\end{table}

Next, we consider a larger discontinuity in the diffusion coefficient by choosing $(\beta^-,\beta^+) = (1,10000)$. The nonsymmetric PPIFE method is used for spatial discretization in the experiment. We choose the penalty parameter $\sigma_B^0 = 1$ again for this large discontinuity case, since the coercivity bound is valid for any positive $\sigma_B^0$.  Table \ref{table: NPPIFE error 1 10000 tau h BE exp} and Table \ref{table: NPPIFE error 1 10000 tau h CN exp} contain errors in backward Euler and Crank-Nicolson methods, respectively. Again, we observe that errors in semi-$H^1$ norm have an optimal convergence rate through mesh refinement for both methods. The convergence rate in $L^2$ norm is second order for Crank-Nicolson and first order for backward Euler. For symmetric PPIFE methods, we have observed similar behavior to the nonsymmetric methods provided that the penalty parameter $\sigma_B^0$ is large enough.

\bibliographystyle{abbrv}

\end{document}